\documentclass[11pt,reqno]{amsart}
\usepackage[textheight = 9in, lmargin=1in,rmargin=1in]{geometry}

\usepackage[utf8]{inputenc}
\usepackage{amstext,latexsym,amsbsy,amsmath,amssymb,amsthm,mathtools,relsize,geometry,enumerate}
\usepackage{hyperref}
\hypersetup{
  colorlinks   = true, 
  urlcolor     = blue, 
  linkcolor    = red, 
  citecolor   = red 
}
\theoremstyle{definition}

\usepackage{graphicx}

\setkeys{Gin}{width=\linewidth,totalheight=\textheight,keepaspectratio}
\graphicspath{{./images/}}

\usepackage{multicol}
\usepackage{booktabs}
\usepackage{mathrsfs}\usepackage{color}
\usepackage{soul}

\newcommand{\rbb}{\mathbb{R}}

\newcommand{\cbb}{\mathbb{C}}

\newcommand{\CM}{\mathcal{CM}}
\newcommand{\Sc}{\mathcal{S}}
\newcommand{\F}[1]{\mathcal{F}\left[#1\right]}

\newcommand{\Kcos}{\mathcal{K}_{\cos}}
\newcommand{\Ksin}{\mathcal{K}_{\sin}}

\newcommand{\la}{\langle}
\newcommand{\ra}{\rangle}
\newcommand{\dom}{\text{Dom}}

\newcommand{\xbf}{\mathbf{x}}
\newcommand{\ybf}{\mathbf{y}}
\newcommand{\Fbf}{\mathbf{F}}
\newcommand{\CMb}{\mathcal{CM}_b}

\renewcommand{\Im}{\text{Im}}
\renewcommand{\Re}{\text{Re}}

\renewcommand{\d}{\mathrm{d}}

\newcommand{\domain}{\mathcal{O}}

\newcommand{\f}{\varphi}

\newcommand{\ep}{\epsilon}

\newcommand{\Enone}[1]{\mathbb{E}#1}
\newcommand{\E}[1]{\mathbb{E}#1}

\newcommand{\close}{\!\!\!}

\newcommand{\sm}{\setminus\{0\}}
\numberwithin{equation}{section}
\theoremstyle{definition}
\theoremstyle{plain}
\newtheorem{theorem}{Theorem}[section]
\newtheorem{corollary}[theorem]{Corollary}

\newtheorem{lemma}[theorem]{Lemma}
\newtheorem{assumption}[theorem]{Assumption}
\newtheorem{proposition}[theorem]{Proposition}
\newtheorem{definition}[theorem]{Definition}

\newtheorem{remark}[theorem]{Remark}

\numberwithin{equation}{section}

\title[Regularity of a linear stochastic equation with memory]{On the H\"older regularity of a linear stochastic partial-integro-differential equation with memory}

\author[S.A.~McKinley]{Scott A.~McKinley$^1$}
\author[H.D.~Nguyen]{Hung D.~Nguyen$^2$}

\thanks{\noindent \hspace{-0.52cm} $^1$ Department of Mathematics, Tulane University, New Orleans, Louisiana, USA, 70118.\\
$^2$ Department of Mathematics, University of California, Los Angeles, California, USA, 90095. }

\begin{document}

\maketitle

\begin{abstract}
In light of recent work on particles fluctuating in linear viscoelastic fluids, we study a linear stochastic partial--integro--differential equation with memory that is driven by a stationary noise on a bounded, smooth domain. Using the framework of generalized stationary solutions introduced in~\cite{mckinley2018anomalous}, we provide conditions on the differential operator and the noise to obtain the existence as well as H\"older regularity of the stationary solutions for the concerned equation. As an application of the regularity results, we compare to analogous classical results for the stochastic heat equation. When the 1d stochastic heat equation is driven by white noise, solutions are continuous with space and time regularity that is H\"older $(1/2-\ep)$ and $(1/4-\ep)$ respectively. When driven by colored--in--space noise, solutions can have a range of regularity properties depending on the structure of the noise. Here, we show that the particular form of colored--in--time memory that arises in viscoelastic diffusion applications, satisfying what is called the Fluctuation--Dissipation relationship, yields sample paths that are H\"older $(1/2-\ep)$ and $(1/2-\ep)$ in space and time.
\end{abstract}


\section{Introduction}

 Let $\domain$ be a bounded domain in $\rbb^d$, $d\geq 1$ and denote by $H=L^2(\domain)$, the Hilbert space of square--integrable functions on $\domain$. Given a self--adjoint negative operator $A:D(A)\subset H\to H$ and a memory function $K:\rbb\to[0,\infty)$, we are interested in the following equation for $u(t,\xbf):\rbb\times\domain\to\rbb^d$
\begin{equation} \label{eq:spatial-gle}
\dot{u}(t,\xbf) = \int_{-\infty}^t\!\!\!K(t-s)Au(s,\xbf)\d s+\Fbf(t,\xbf),
\end{equation}
where $\Fbf(t,\xbf)$ has the representation
\begin{align}\label{form:F}
\Fbf(t,\xbf)=\sum_{k\geq 1} \lambda_k F_k(t)e_k(\xbf).
\end{align}
Here, $\{\lambda_k\}_{k \ge 1}$ is a sequence of positive constants, $\{e_k\}_{k\ge 1}$ is an orthonomal basis for $H$, and $\{F_k(t)\}_{k\geq 1}$ is a sequence of i.i.d.~stationary Gaussian processes with autocovariance 
\begin{equation} \label{eq:memory}
\E[F_k(t)F_k(s)]=K(|t-s|).
\end{equation}
The goal of this note is to give an analysis on the well--posedness and regularity of~\eqref{eq:spatial-gle} under appropriate assumptions on $A$ and $K$. Our motivation for ~\eqref{eq:spatial-gle}--\eqref{form:F}--\eqref{eq:memory} is as follows. Suppose that the orthonormal basis $\{e_k\}$, as in~\eqref{form:F}, diagonalizes $A$ in the sense that for each $k \in \mathbb{N}$, there exists an $\alpha_k>0$ such that $Ae_k=-\alpha_ke_k$. By writing $u(t,\xbf)=\sum_{k\geq 1}u_k(t)e_k(\xbf)$, formally, $u_k(t)$ satisfies the following equation
\begin{equation}\label{eq:spatial-gle:mode}
\dot{u}_k(t) = -\alpha_k\int_{-\infty}^t \!\!\! K(t-s)u_k(s)\d s+\lambda_k F_k(t).
\end{equation} 
Equation~\eqref{eq:spatial-gle:mode} is a simpler version of the so--called \emph{generalized Langevin} equation (GLE) that is used to model single--particle movements in viscoelastic fluids. The full 1d GLE has the following form \cite{hohenegger2017equipartition, hohenegger2017fluid, mckinley2018anomalous}
\begin{equation} \label{eq:gle:full}
\begin{cases}
m\dot{v}(t) = -\gamma v(t)-\beta\int_{-\infty}^t K(t-s)v(s)\d s + \lambda F(t)+ \sqrt{2\gamma}\dot{B}(t),\\
\Enone(F(t)F(s))=K(t-s),
\end{cases}
\end{equation}
where $m$ is the mass, $\gamma$ is the drag due to viscosity, $F(t)$ is a stationary Gaussian process and $B(t)$ is a standard Brownian motion. Here the covariance of the stationary force $F$ is the same as the memory kernel $K$, which is a force--balance condition resulting from the \emph{Fluctuation-Dissipation relationship} \cite{mori1965continued}. Historically, the GLE was first proposed and studied in the work of~\cite{kubo1966fluctuation,mori1965continued} and later popularized in \cite{mason1995optical}. The GLE has been given renewed attention in the last two decades due to its ability to produce what is known as \emph{anomalous diffusion} \cite{morgado2002relation}. Namely, let $x(t):=\int_0^t v (s)\d s$ where $v(t)$ satisfies \eqref{eq:gle:full}. If the memory kernel $K$ is integrable, it has been shown that the Mean-Squared Displacement $\E[x(t)^2]$ grows linearly in time. On the other hand, if there exists an $\alpha\in(0,1]$ such that $K(t)\sim t^{-\alpha}$ as $t\to\infty$, then for $\alpha\in(0,1)$, $\E[x(t)^2]\sim t^\alpha$ \cite{kou2008stochastic, mckinley2018anomalous} and for $\alpha=1$, $\E[x(t)^2]\sim t/\log(t)$ as $t\to\infty$ \cite{didier2020asymptotic}. Here $f(t)\sim g(t),\ t\to\infty$ means $\lim_{t\to\infty}f(t)/g(t)=c\in(0,\infty)$. While the GLE is useful when modeling single--particle movements, it fails to capture multi--particle interactions through fluctuating hydrodynamics. Recently, a first step in this direction was a model proposed and studied numerically in \cite{hohenegger2017fluid} that generalized the fluctuating Landau-Lifschitz Navier-Stokes equations from viscous to viscoelastic fluids. The model that we consider \eqref{eq:spatial-gle} is the linearized version of the system, which appears in \cite{hohenegger2017fluid}, without fluid specific terms like the Navier--Stokes non-linear term or time-dependent pressure. Inspired by~\cite{didier2020asymptotic,hohenegger2017fluid,mckinley2018anomalous}, we would like to investigate the regularity of the \emph{stationary} solutions for~\eqref{eq:spatial-gle} whenever they make sense.

Stochastic partial-integro-differential equations with infinite delay have been studied in literature \cite{bonaccorsi2012asymptotic,
clement1988existence,clement1998white,
weinan2002gibbsian,weinan2001gibbsian,ito1964stationary,
 yamazaki2019gibbsianB,yamazaki2019gibbsianA}. Recently, in \cite{bonaccorsi2012asymptotic}, when $A$ is the usual Laplace operator, the author considered a stochastic heat equation with memory that has the form
\begin{align} \label{eqn:stochastic-da-prato}
\dot{u}(t,\xbf)  = k_0A\,u(t,\xbf)-\int_{-\infty}^t \close K(t-s)A\, u(s,\xbf)\d s+\dot{W}(t,\xbf),\quad (t,\xbf)\in\rbb\times\domain,
\end{align}
where $k_0>0$ is a constant satisfying $k_0>\int_0^\infty \! K(s)ds$ and $W$ is a Wiener process. Equation~\eqref{eqn:stochastic-da-prato} arises from theory of thermal viscoelasticity, in which $k_0$ and $K$ respectively represent the instantaneous
conductivity and the heat flux memory kernel. See also~\cite{bonaccorsi2004large,bonaccorsi2006infinite,clement1996some,clement1997white,
clement1998white} for related work on stochastic systems similar to~\eqref{eqn:stochastic-da-prato}. It is interesting to study how and whether the H\"older regularity of solutions to \eqref{eqn:stochastic-da-prato} differ from the analogous stochastic heat equation that does not feature memory: 
\begin{align}  \label{eqn:stochastic-heat}
\dot{u}(t,\xbf)  = A\,u(t,\xbf)+\dot{W}(t,\xbf),\qquad (t,\xbf)\in\rbb\times\domain.
\end{align}
In fact, under suitable assumptions on $W$ and $A$, H\"older regularity of solutions~\eqref{eqn:stochastic-da-prato} does not differ from that of~\eqref{eqn:stochastic-heat} \cite{da2014stochastic, hairer2009introduction}. The proofs of these arguments rely on Kolmogorov's criterion and estimates on the differences in space and time within a solution. Motivated by these results, we would like to investigate how the memory structure intrinsic in ~\eqref{eq:spatial-gle} affects regularity of its solutions when compared to the analogous versions of \eqref{eqn:stochastic-da-prato} and \eqref{eqn:stochastic-heat}.

Our problem differs from those considered previously as follows: first, in preceding works, the random part is usually a Wiener process that is white-in-time and is defined on an auxiliary Hilbert space that does not depend on the memory. By contrast, our noise $\Fbf$ has the decomposition~\eqref{form:F} and each mode $F_k$ is related with the memory $K$ via relation~\eqref{eq:memory}, which follows from the Fluctuation--Dissipation relationship. Thus, one can regard $\Fbf$ in~\eqref{eq:spatial-gle} as being colored-in-time. In addition, the memory kernels in previous works \cite{bonaccorsi2012asymptotic,clement1988existence,clement1998white} are required to be integrable on the positive real line. This condition excludes those that have a slow power-law decay, namely $K(t)\sim t^{-\alpha}$ as $t\to\infty$ for $\alpha\in (0,1]$. As already mentioned, the Mean--Squared Displacement $\E[x(t)^2]$ associated with these kernels in~\eqref{eq:gle:full} obeys a sub-linear growth, i.e., $t^\alpha$ or $t/\log t$ as $t\to\infty$, respectively when $\alpha\in(0,1)$ \cite{mckinley2018anomalous} or $\alpha=1$ \cite{didier2020asymptotic}. In our work, the memory kernels need not be integrable and moreover, the conditions that we impose allow them to have a slow power--law decay, cf.~Assumption~\ref{cond:K}. To be more precise, throughout the rest of this work, we will restrict our memory kernel to the class of \emph{completely monotone} functions, which have been studied in great detail elsewhere \cite{
bonaccorsi2012asymptotic,clement1996some,clement1998white}. The important feature these functions possess is that their Fourier transforms admit clean expressions, which allows us to perform analysis on the regularity of~\eqref{eq:spatial-gle} in Section~\ref{sec:main-results}. As a result of these assumptions, the system~\eqref{eq:spatial-gle} admits better regularity compared to~\eqref{eqn:stochastic-da-prato}--\eqref{eqn:stochastic-heat}. In particular, the one-dimensional versions of~\eqref{eqn:stochastic-da-prato}--\eqref{eqn:stochastic-heat} with white noise are both $\gamma$-H\"older continuous in time for $\gamma\in (0,1/4)$~\cite{bonaccorsi2012asymptotic,da2014stochastic}. In contrast, as a corollary of Theorem~\ref{thm:spatial-gle:Holder} below, our 1D heat equation with colored noise $\Fbf$ is $\gamma-$H\"older continuous in time for $\gamma\in(0,1/2)$. 
 
Finally, we mention that on the deterministic side, there are many related works on partial differential equations with memeory kernels. For a few examples, we refer the reader to~\cite{barbu1975nonlinear,barbu1976nonlinear,
chekroun2012invariant,
conti2020nonclassical,
conti2006singular,gatti2005navier}.

The rest of the paper is organized as follows. In Section~\ref{sec:result}, we introduce our definition of a stationary solution for~\eqref{eq:spatial-gle}, state our assumptions, and summarize our main results, particularly on the H\"older regularity of ~\eqref{eq:spatial-gle}, cf.~Theorem~\ref{thm:spatial-gle:Holder}. In Section~\ref{sec:prelim}, we present the preliminaries needed for our analysis, e.g.,~weak formulations of~\eqref{eq:spatial-gle:mode} and a Fourier analysis on the class of completely monotone functions. We then detail the proofs of main results in Section~\ref{sec:main-results}. We finish with Section~\ref{sec:discuss} discussing related problems as well as applications to stochastic heat equations with memory.

\section{Assumptions and main results}  \label{sec:result}
 
Concerning the linear operator $A$, we assume the following standard condition (see also \cite[Section 5.5.1]{da2014stochastic}).
\begin{assumption} \label{cond:A} The orthonormal basis $\{e_k\}_{k\ge 1}$ in~\eqref{form:F} belongs to $D(A)$ and diagonalizes $A$, i.e., there exists an increasing sequence $0<\alpha_1\leq \alpha_2\leq \ldots$ such that $\{\alpha_k\}\uparrow\infty$ and that
\begin{displaymath}
Ae_k=-\alpha_k e_k,\quad k\geq 1.
\end{displaymath}
We further assume that
\begin{displaymath}
|e_k(\xi)|\leq M c_k,\quad |\nabla e_k(\xi)|\leq M\alpha_k^{1/2}c_k, \quad\forall\xi\in\domain,\quad k\geq 1,
\end{displaymath}
for some positive constant $M$ and a sequence of (increasing) positive numbers $\{c_k\}_{k\ge 1}$.
\end{assumption}

One can think of $A$ as the usual Laplacian operator, but it need not be. Up to this point, we have not defined what we mean by a stationary solution of~\eqref{eq:spatial-gle:mode}. We shall postpone introducing the formulation of these solutions, cf.~Definition~\ref{def:GLE:weak-solution}, as well as the required elements to construct them in Section~\ref{sec:GLE:weak-solution}, based on the work of~\cite{didier2020asymptotic,ito1954stationary,mckinley2018anomalous}. We now state the following important definition of a {\it stationary solution} of~\eqref{eq:spatial-gle}.
\begin{definition}\label{def:spatial-gle:solution} A random field $u(t,\xbf)=\sum_{k\geq 1}u_k(t)e_k(\xbf)$ is called a stationary solution of~\eqref{eq:spatial-gle} if 

\noindent (a) for all $t\in\rbb$, $u(t,\cdot)\in L^2(\Omega;H)$, i.e., $\Enone\|u(t,\cdot)\|^2_H<\infty$; and

\noindent (b) the process $u_k(t)$ satisfies $u_k(t)=\la V_k,\delta_{t}\ra $ where $\{V_k\}_{k\ge 1}$ are mutually independently weak stationary solutions of~\eqref{eq:spatial-gle:mode} in the sense of distribution given by Definition~\ref{def:GLE:weak-solution}, $\delta_t$ is the Dirac function centered at $t$, and $\la V_k,\delta_{t}\ra$ denotes the action of $V_k$ when applied to $\delta_t$.
\end{definition}

With regard to the memory kernel $K$, we first state the definition of the class of \emph{completely monotone} functions, of which we will assume that $K$ is a member.
\begin{definition}\label{def:complete-monotone}
A function $K:(0,\infty)\to[0,\infty)$ is completely monotone if $K$ is of class $C^\infty(0,\infty)$ and $(-1)^n K^{(n)}(t)\geq 0$ for all $n\geq 0$, $t>0$.
\end{definition}

Throughout this note, unless stated otherwise, we require that the memory kernel $K$ satisfy the following assumption.
\begin{assumption} \label{cond:K} We assume that $K:\rbb\to[0,\infty)$ is symmetric and eventually decreasing to zero as $t\to\infty$. Furthermore, $K(t)$ is completely monotone on $t\in (0,\infty)$.
\end{assumption}

The class of completely monotone functions includes a special case of memory kernels for which $K$ can be be expressed as a sum of exponentials~  \cite{goychuk2012viscoelastic,goychuk2014molecular,hohenegger2017fluid,morgado2002relation}. In addition, Assumption~\ref{cond:K} allows us to include power--law decay functions, e.g., $(1+|t|)^{-\alpha}$, $\alpha>0$ \cite{mckinley2018anomalous}. The advantage of completely monotone functions is that their Fourier transform can be explicitly computed (cf. Lemma~\ref{lem:complete-monotone:fourier}) and thus is helpful for our analysis. We will see later in Section~\ref{sec:main-results} that their Fourier structures play an important role when we investigate the well--posedness and regularity of~\eqref{eq:spatial-gle}.

With regards to the parameters $\lambda_k$ and $\alpha_k$, we first assume that they satisfy the following condition. See also \cite[Hypothesis 2.10]{bonaccorsi2012asymptotic} and \cite[Condition (5.40)]{da2014stochastic}.
\begin{assumption} \label{cond:alpha-k:wellposed}
Let $\{\alpha_k\}_{k\geq 1}$ be as in Assumption~\ref{cond:A} and $\{\lambda_k\}_{k\geq 1}$ be as in~\eqref{form:F}. We assume that
\begin{align*}
\sum_{k\geq 1}\frac{\lambda_k^2}{\alpha_k}<\infty.
\end{align*}
\end{assumption}

In Theorem~\ref{thm:spatial-gle:well-posed} below, we will see not only that Assumption~\ref{cond:alpha-k:wellposed} be necessary, but it is also sufficient to guarantee the existence of weak stationary solutions of~\eqref{eq:spatial-gle}. Moreover, we note that Assumption~\ref{cond:alpha-k:wellposed} is only about the pairs $\{(\alpha_k,\lambda_k)\}_{k \geq 1}$ and does not require information on $\{c_k\}_{k\ge 1}$ as in Assumption~\ref{cond:A}. Concerning H\"older continuity, we assume the following further condition on the triples $\{(c_k,\alpha_k, \lambda_k)\}_{k \geq 1}$.
\begin{assumption} \label{cond:alpha-k:Holder}
Let $\{\alpha_k\}_{k\geq 1}$, $\{c_k\}_{k\ge 1}$ be as in Assumption~\ref{cond:A} and $\{\lambda_k\}_{k\geq 1}$ be as in~\eqref{form:F}. There exists a constant $\eta\in(0,1)$ such that
\begin{align*}
\sum_{k\geq 1}\frac{\lambda_k^2c_k^2}{\alpha_k^\eta}<\infty.
\end{align*}
\end{assumption}

\begin{remark}
Since $c_k>0$ and $\alpha_k$ are non--decreasing by Assumption~\ref{cond:A}, it is clear that Assumption~\ref{cond:alpha-k:Holder} implies Assumption~\ref{cond:alpha-k:wellposed}. We note however that Assumption~\ref{cond:alpha-k:Holder} is also standard and can be found in the literature of linear SPDEs, see for instance~\cite[Lemma 3.27]{bonaccorsi2012asymptotic} and~\cite[Lemma 5.21]{da2014stochastic}.
\end{remark}

  We now state our first important result, giving the existence of stationary solutions for~\eqref{eq:spatial-gle}.
\begin{theorem}[Well--posedness]\label{thm:spatial-gle:well-posed} Suppose that Assumptions~\ref{cond:A} and \ref{cond:K} hold. Then, Equation~\eqref{eq:spatial-gle} admits stationary solutions $u(t,\xbf)$ in the sense of Definition~\ref{def:spatial-gle:solution} if and only if Assumption~\ref{cond:alpha-k:wellposed} is satisfied.
\end{theorem}

\begin{remark}
We note that in Theorem~\ref{thm:spatial-gle:well-posed}, we do not address the pathwise uniqueness as commonly found in SPDEs. Since stationary processes are characterized by their autocovariance function and spectral densities~\cite{cramer2013stationary,ito1954stationary}, later on in Section~\ref{sec:main-results}, we will see that stationary solutions $u(t,\xbf)$ of~\eqref{eq:spatial-gle} admit the unique covariance representation
$$\Enone[u(t,\xbf)u(s,\ybf)]=\sum_{k\ge 1}\int_\rbb e^{i(t-s)\omega }\rho_k(\omega)\d\omega \,e_k(\xbf)e_k(\ybf),$$
where $\rho_k(\omega)$ given by~\eqref{eq:spectral-density} is the spectral density of $u_k(t)$ as in Proposition~\ref{prop:spatial-gle:mode:process} below. See also Proposition~\ref{prop:spatial-gle:mode:weak-solution}.
\end{remark} 

The proof of Theorem~\ref{thm:spatial-gle:well-posed} makes use of a Fourier analysis on the memory kernels and will be carried out in Section~\ref{sec:main-results}. In particular, we will generalize results in~\cite{hohenegger2017equipartition,hohenegger2017fluid}, where $K$ has the form as a sum of exponentials, to calculate explicitly the second moment of $u_k(t)$ as in the decomposition $u(t,\xbf)=\sum_{k\ge 1}u_k(t)e_k(\xbf)$.

 We now state the main result of the paper on the regularity of~\eqref{eq:spatial-gle}.
\begin{theorem}[Time and space regularity]\label{thm:spatial-gle:Holder} Suppose that Assumptions~\ref{cond:A}, \ref{cond:K} and \ref{cond:alpha-k:Holder} are satisfied. Let $u(t,\xbf)$ be the solution of~\eqref{eq:spatial-gle} as in Theorem~\ref{thm:spatial-gle:well-posed} and $\eta$ be the constant from Assumption~\ref{cond:alpha-k:Holder}. Then there exists a modification $U(t,\xbf)$ of $u(t,\xbf)$ such that $U$ is $\gamma-$H\"older continuous in time and space for any $\gamma\in(0,1-\eta)$.
\end{theorem}
The proof of Theorem~\ref{thm:spatial-gle:Holder} carried out in Section~\ref{sec:main-results} will employ the classical Kolmogorov criterion on space--time regularity of random fields, which in turn relies heavily on delicate estimates on difference of solutions of~\eqref{eq:spatial-gle}. In Section~\ref{sec:discuss}, we will detail application of Theorem~\ref{thm:spatial-gle:Holder} to stochastic heat equations. In particular, one will see that while space regularity remains the same, time regularity of~\eqref{eq:spatial-gle} is smoother than those of~\eqref{eqn:stochastic-da-prato} and \eqref{eqn:stochastic-heat}.

\section{Mathematical Preliminaries}  \label{sec:prelim}

Throughout the rest of the paper, we will use $C$ and $c$ to denote generic positive constants, whose values may change from one line to the next. When the dependence on parameters is important, this will be indicated in parenthesis, e.g., $c(\alpha,q)$ depends on parameters $\alpha$ and $q$.
\subsection{Weak solutions of the GLE} \label{sec:GLE:weak-solution} In order to define weak solutions of~\eqref{eq:spatial-gle:mode}, for the reader's convenience, we briefly review the framework in~\cite{mckinley2018anomalous}. For given $\lambda,\,\beta>0$, we consider a stochastic process $v:\rbb\to\rbb$ governed by the formal stochastic integro--differential equation
\begin{equation} \label{eq:GLE}
\dot{v}(t)=-\beta\int_{-\infty}^t \! \! \! \! K(t-s)v(s)\d s+\lambda F(t),
\end{equation}
where $F(t)$ is a stationary Gaussian process whose autocovariance function is $K(t)$. Our methods rely heavily on Fourier analysis for spectral functions and so for a function $f:\rbb\to\cbb$, we define the Fourier transform of $f$ in the sense of improper integrals as
\begin{displaymath}
\widehat{f}(\omega)=\int_\rbb f(t) e^{-it\omega} \d t.
\end{displaymath}
If a function $K(t)\in L^1_{\text{loc}}(\rbb)$ is symmetric and eventually decreasing to zero as $t\to\infty$, then the above integral is well--defined for every $\omega\neq 0$ \cite{soni1975slowly}. Let $\Sc$ be the class of Schwarz functions defined on $\rbb$ and $\Sc'$ be the space of tempered distributions on $\Sc$. The action of $f\in\Sc'$ on $\f\in\Sc$ is denoted by $\la f,\f\ra$. Also, the Fourier transform $\F{f}$ of $f\in\Sc'$ in the sense of distributions is defined as follows:
\begin{align*}
\forall\f\in\Sc,\qquad\la \F{f},\f\ra:=\la f,\widehat{\f} \ra .
\end{align*}
It is well known that the Fourier map $\mathcal{F}:\Sc'\to\Sc'$ is a one--to--one relation. We begin with the definition of a (weak) stationary process \cite{cramer2013stationary}.
\begin{definition} \label{def:stationary-process}
 A stochastic process $\{F(t)\}_{t \in \rbb}$ is mean--square continuous and stationary if for all $t, s\in\rbb$,
		\begin{enumerate}[(a)]
		\item $\E\lvert F(t) \rvert^2<\infty$ and $\lim_{h\to 0}\E\lvert F(t+h)-F(t) \rvert^2 =0$;
		\item $\E[F(t)]=a$, for some constant $a$ (we may assume $a=0$); and
		\item the covariance function $\E[F(t)\overline{F(s)}]$ depends only on the difference $(t-s)$.
		\end{enumerate}
\end{definition}

By Bochner's Theorem, a stationary process $F$ is characterized by a positive definite function $r$ and a finite Borel measure $\nu$ such that the following holds for all $t,\,s\in\rbb$ \cite{cramer2013stationary}: 
\begin{align}\label{eq:stationary-process}
\Enone\big[F(t)\overline{F(s)}\big]=r(t-s)=\int_\rbb e^{i(t-s)y}\nu(\d y).
\end{align}
In the above, $r$ is called the \emph{covariance} and $\nu$ is called the \emph{spectral measure} of $F$. Furthermore, there always exists a modified version $\widetilde{F}$ of $F$ such that $\widetilde{F}$ is Gaussian \cite{cramer2013stationary}. A generalization of stationary processes is the class of stationary random distributions, introduced by It\^{o}'s seminal work \cite{ito1954stationary}. Denote by $\tau_h$, the shift transform on $\Sc$, $\tau_h\varphi(x):=\varphi(x+h)$.
\begin{definition} A linear functional $F:\Sc\to L^2(\Omega)$, the space of all complex--valued random variables with finite variance, is called a {\it stationary random distribution} on $\Sc$ if for all $h\in\rbb$, $\varphi_1,\varphi_2\in\Sc$,
\begin{align*}
\E\big[\langle  F,\tau_h \varphi_1\rangle \overline{\langle F,\tau_h\varphi_2  \rangle}\big]=\E\big[\langle  F,\varphi_1\rangle \overline{\langle F,\varphi_2  \rangle}\big].
\end{align*}
\end{definition}

Next, we recall the definition of a positive definite tempered distribution $f$ \cite{ito1954stationary}, which is analogous to the positive definiteness of real--valued functions.
\begin{definition} \label{def:temper-distribution:positive-def} A tempered distribution $f\in\Sc'$ is called {\it positive definite} if for any $\varphi\in \Sc$,
\begin{align*}
 \langle f,\varphi*\widetilde{\varphi}\rangle \geq 0, 
\end{align*}
where $\widetilde{\varphi}(x)=\varphi(-x)$.
\end{definition}

A stationary random distribution $F$ is characterized by its associated positive definite tempered distribution $r$ and a non--negative measure $\nu$ in the following sense: (\cite{ito1954stationary})
\begin{align} \label{eq:stationary-random-distibution}
\forall\f_1,\,\f_2\in\Sc,\qquad\Enone\Big[\la F,\f_1\ra\overline{\la F,\f_2\ra}\Big] = \la r,\f_1*\widetilde{\f_2}\ra= \int_\rbb \overline{\widehat{\f_1}(y)}\widehat{\f_2}(y)\nu(\d y),
\end{align}
where $\nu$ satisfies
\begin{equation}\label{ineq:spectral-measure}
	\int_\rbb \frac{\nu( \d x  )}{(1+x^2)^k}<\infty,
\end{equation}
for some integer $k$. Furthermore, if $\nu$ is absolutely continuous with respect to Lebesgue measure on $\rbb$, then we are able to extend the above stationary random distribution $F$ to a \emph{generalized random distribution} $\Phi:\dom(\Phi)\subset\Sc'\to L^2(\Omega)$ such that the following holds
\begin{align} \label{eq:generalized-random-distribution}
\Enone\Big[\la \Phi, f_1\ra \overline{\la \Phi,f_2\ra}\Big] =  \int_\rbb\overline{\F{f_1}(y)}\F{f_2}(y)\nu( \d y ).
\end{align}
Here, the domain of $\Phi$ consists of those tempered distributions $f$ such that $\F{f}\in L^2(\nu)$ where 
\begin{align*}
L^2(\nu)=\big\{g:\rbb\to\cbb,\int_\rbb |g(y)|^2\nu( \d y )<\infty\big\}.
\end{align*}
Comparing~\eqref{eq:generalized-random-distribution} with~\eqref{eq:stationary-random-distibution}, we readily see that $\Phi$ is an extension of $F$ since $\Sc\subset\dom(\Phi)$, see~\cite{mckinley2018anomalous} for a more detailed discussion. We now can define the function--valued version of $\Phi$ as follows.
\begin{definition} \label{def:velocity} Let $\Phi$ be a random operator associated with a measure $\nu$ as in~\eqref{eq:generalized-random-distribution}. Let $\delta_t$ be the Dirac $\delta-$distribution centered at $t$. If $\delta_t$ and $1_{[0,t]}$ are in $\emph{\dom}(\Phi)$, then we define
\begin{align*}
v(t):=\la \Phi,\delta_t\ra,\quad\text{ and }\quad x(t):=\la \Phi,1_{[0,t]}\ra.
\end{align*}  
We note that $x(t)$ can be defined without $v(t)$.
\end{definition}

In view of~\cite[Lemma 2.17]{mckinley2018anomalous}, $\delta_t\in\dom(\Phi)$ if and only if the representation measure $\nu$ is finite. Moreover $v(t)$ is a stationary Gaussian process, which is consistent with~\eqref{eq:stationary-process}. We now turn our attention to weak solutions of~\eqref{eq:GLE}. Formally, we can multiply both sides of~\eqref{eq:GLE} by a Schwarz function $\f$, then perform an integration by parts on the left--hand side and a change of variable on the convolution term on the right--hand side to arrive at
\begin{align*}
-\int_\rbb v(t)\f'(t)\d t=-\beta\int_\rbb v(t)\int_\rbb K^+(y)\f(t+y) \d y \d t+\lambda\int_\rbb F(t)\f(t)\d t,
\end{align*}
where $K^+(t):=K(t)1_{\{t\geq 0\}}$. Bringing the convolution to the left--hand side now yields
\begin{align*}
\int_\rbb v(t)\big[-\f'(t)+\beta\widetilde{ K^+*\widetilde{\f} }(t)\d t  \big]\d t=\lambda\int_\rbb F(t)\f(t)\d t.
\end{align*}
Having introduced the notion of generalized random distributions, we rewrite the above equation in the following weak form
\begin{align}
\la V,-\f'+\beta \widetilde{ K^+*\widetilde{\f} }\ra = \lambda\la F,\f\ra,
\end{align}
where $\widetilde{f}(x):=f(-x)$. In the above, the LHS is understood as an action of a generalized random distribution $V$ applied to an element in its domain, whereas the RHS is the usual action of the stationary random distribution $F$ applied to $\f\in\Sc$. Furthermore, the random distribution $F$ is characterized by its covariance structure
\begin{align}\label{form:F:random-distribution}
\forall\f_1,\f_2\in\Sc,\qquad \Enone\Big[\la F,\f_1\ra\overline{\la F,\f_2\ra}\Big]=  \int_\rbb K(t)(\f_1*\widetilde{\f_2})(t)\d t.
\end{align}
\begin{remark} In general, a real--valued function $K$ can be regarded as a tempered distribution by setting
\begin{align*}
\la K,\f\ra:=\int_\rbb K(t)\f(t)\d t,
\end{align*}
for any $\f\in\Sc$. The above integral always converges as long as $K$ belongs to $ L^1_{\text{loc}}(\rbb)$ and does not grow exponentially fast \cite{strichartz2003guide}. In addition, if $K$ satisfies Assumption~\ref{cond:K}, then $K$ is indeed a positive definite tempered distribution \cite{mckinley2018anomalous}, which in turn implies that $F$ as in~\eqref{form:F:random-distribution} is a stationary random distribution.
\end{remark}
 
With the above observation, we have the following definition of a weak solution of~\eqref{eq:GLE}.
\begin{definition} \label{def:GLE:weak-solution}
\cite{mckinley2018anomalous} Let $\nu$ be a non-negative measure satisfying condition \eqref{ineq:spectral-measure} and $V$ be the operator associated with $\nu$ via the relation~\eqref{eq:generalized-random-distribution}. Then $V$ is a \emph{weak stationary solution} for Equation \eqref{eq:GLE} if $V$ satisfies the following conditions.
\begin{enumerate}[(a)]
\item For all $\varphi\in\Sc$, $K^+*\varphi$ belongs to $\text{Dom}(V)$.
\item For any $\varphi,\psi\in \Sc$, it holds that
\begin{align*}
\Enone\bigg[\langle V,-\varphi'+\beta\widetilde{K^+*\widetilde{\varphi}}\rangle\overline{\langle V,-\psi'+\beta\widetilde{K^+*\widetilde{\psi}}\rangle}\bigg]=\lambda^2\E\big[\langle F,\varphi\rangle \overline{\langle F,\psi\rangle}\big].
\end{align*}
\end{enumerate}
\end{definition}
Using this definition, we will address the well--posedness of~\eqref{eq:spatial-gle:mode} in Section~\ref{sec:main-results}.

\subsection{Completely Monotone Functions}\label{sec:complete-monotone} In this subsection, we collect several properties of the class of \emph{completely monotone functions} that are needed for the analysis of the regularity of stationary solutions of~\eqref{eq:spatial-gle}. It is known that such a function $K$ can be characterized in terms of the Laplace transform of Radon measures.
\begin{theorem}[Hausdorff--Bernstein--Widder Theorem]\label{thm3} A function $K$ is completely monotone if and only if $K$ admits the formula
\begin{equation} \label{eq:complete-monotone:laplace}
K(t)=\int_0^\infty \!\!\! e^{-tx} \mu (\d x  ),
\end{equation}
where $\mu$ is a positive measure on $[0,\infty)$. 
\end{theorem}

It is convenient to denote
\begin{enumerate}[(a)]
\item $\mathcal{CM}\label{CM}$, the class of all \textit{completely monotone} functions; and
\item $\mathcal{CM}_b\label{CMb}$, the class of all $K\in \mathcal{CM}$ such that the measure $\mu$ in~\eqref{eq:complete-monotone:laplace} is finite.
\end{enumerate} 
\begin{remark} \label{rem:CMb}
 Notice that if $K\in \mathcal{CM}_b$, by setting $K(0):=\int_0^\infty \mu( \d x  )=\mu([0,\infty))$, $K$ can be extended to be continuous on $[0,\infty)$. Hence, $\CM_b=C[0,\infty)\cap\CM$. In view of Assumption~\ref{cond:K}, the class of kernels that we consider is a subset of $\mathcal{CM}_b$.
\end{remark}
We now turn to Fourier transforms of $\CMb$ functions.
\begin{lemma} \label{lem:complete-monotone:fourier}
Suppose that $K\in\CM_b$ and that $K$ is decreasing to $0$ as $t\to\infty$. Let $\mu$ be the representation measure as in~\eqref{eq:complete-monotone:laplace}. Then for every $\omega\neq 0$, it holds that
\begin{equation} \label{eq:complete-monotone:fourier}
\int_0^\infty\close K(t)e^{-it\omega}\d t =\int_0^\infty\close \frac{x}{x^2+\omega^2}\mu( \d x  )-i\int_0^\infty\close \frac{\omega}{x^2+\omega^2}\mu( \d x  ).
\end{equation} 
\end{lemma}
\begin{proof} We first note that since $\mu$ is a finite measure, this implies that
\begin{align*}
\int_0^\infty\close \frac{x}{x^2+\omega^2}\mu( \d x  ) \leq \frac{1}{2|\omega|}\int_\rbb\mu( \d x    )< \infty,
\end{align*}
and that
\begin{align*}
\int_0^\infty\close \frac{|\omega|}{x^2+\omega^2}\mu( \d x  ) \leq \frac{1}{|\omega|}\int_\rbb\mu( \d x    )<\infty.
\end{align*}
In addition, $K(t)$ decreasing to 0 as $t\to\infty$ implies that
\begin{equation*}
\mu(\{0\})=\lim_{t\to\infty}\int_0^\infty\close e^{-tx}\mu(\d x )=\lim_{t\to\infty}K(t)=0.
\end{equation*}
Now by the definition of improper integrals,
\begin{align*}
\int_0^\infty\close K(t)e^{-it\omega}\d t= \lim_{A\to\infty}\int_0^A\close  K(t)e^{-it\omega}\d t.
\end{align*}
Substituting $K(t)=\int_0^\infty e^{-tx}\mu( \d x    )$ on the above RHS, we have a chain of implications
\begin{align*}
\int_0^A\close K(t)e^{-it\omega}\d t &= \int_0^A \close\int_0^\infty\close e^{-tx}\mu( \d x    ) e^{-it\omega}\d t \\
&= \int_0^\infty \close\int_0^A\close e^{-(x+i\omega)t}\d t \mu( \d x    )\\
&= \int_0^\infty\!\!\frac{1-e^{-(x+i\omega)A}}{x+i\omega}\mu( \d x    )\\
&= \int_0^\infty\!\! \frac{\left(1-e^{-(x+i\omega)A}\right)x}{x^2+\omega^2}\mu( \d x    )-i\int_0^\infty \close\frac{\left(1-e^{-(x+i\omega)A}\right)\omega}{x^2+\omega^2}\mu( \d x    ).
\end{align*}
Considering the first integral term, it is clear that for all $x\geq 0$ and $\omega\neq 0$, 
\begin{align*}
\frac{\left|1-e^{-(x+i\omega)A}\right|x}{x^2+\omega^2}\leq \frac{2x}{x^2+\omega^2}, \qquad A\to\infty,
\end{align*}
and that
\begin{align*}
\frac{\left(1-e^{-(x+i\omega)A}\right)x}{x^2+\omega^2}\to \frac{x}{x^2+\omega^2}, \qquad A\to\infty.
\end{align*}
By the Dominated Convergence Theorem, we obtain
\begin{align*}
\int_0^\infty\!\! \frac{\left(1-e^{-(x+i\omega)A}\right)x}{x^2+\omega^2}\mu( \d x    )\to \int_0^\infty\!\! \frac{x}{x^2+\omega^2}\mu( \d x    ),\qquad A\to\infty.
\end{align*}
With regard to the second term, we note that $\mu(\{0\})=0$ as reasoned above. It follows that for $\mu-$almost every $x\in[0,\infty)$, we have
\begin{align*}
\frac{\left(1-e^{-(x+i\omega)A}\right)\omega}{x^2+\omega^2}\to \frac{\omega}{x^2+\omega^2}.
\end{align*}
Also, for all $x\geq 0$ and $\omega\neq 0$, 
\begin{align*}
\frac{\left|(1-e^{-(x+i\omega)A})\omega\right|}{x^2+\omega^2}\leq \frac{2x}{x^2+\omega^2},
\end{align*}
which implies
\begin{align*}
\int_0^\infty \close\frac{\left(1-e^{-(x+i\omega)A}\right)\omega}{x^2+\omega^2}\mu( \d x    )\to\int_0^\infty \close\frac{\omega}{x^2+\omega^2}\mu( \d x    ),\qquad A\to\infty,
\end{align*}
by the Dominated Convergence Theorem. The proof is thus complete.
\end{proof}
We finish this subsection by the following important observations on the class of $\CM_b$ that will be the main ingredients for our analysis on the regularity of the solutions of~\eqref{eq:spatial-gle} in Section~\ref{sec:main-results}.
\begin{lemma} \label{lem:complete-monotone:fourier-analysis} Suppose that $K\in\CM_b$ and that $K(t)\downarrow 0$ as $t\to\infty$. Define for $\omega\neq 0$,
\begin{align} \label{form:KcosKsin}
\Kcos(\omega):=\int_0^\infty\close K(t)\cos(t\omega)\d t,\quad\text{and}  \quad\Ksin(\omega):=\int_0^\infty\close K(t)\sin(t\omega)\d t.
\end{align} 
Then, the following properties hold.

	(a) $\Kcos(\omega)$ is decreasing, $\omega\Kcos(\omega)$ is bounded and $\omega^2\Kcos(\omega)$ is increasing on $\omega\in [0,\infty)$.
	
	(b) $\omega\Ksin(\omega)$ is increasing on $\omega\in [0,\infty)$ and $\lim_{\omega\to\infty}\omega\Ksin(\omega)=K(0)$.
	
	(c) The ratio $\frac{\Ksin(\omega)}{\omega}$ is decreasing to 0 as $\omega\to\infty$. Consequently, for $k$ sufficiently large, the equation 
	\begin{align} \label{eq:complete-monotone:fourier:ksin}
		1-\alpha_k\frac{\Ksin(\omega)}{\omega}=0,
	\end{align}
has a unique solution $\omega_k$ on $(0,\infty)$ where $\alpha_k$ is as in Assumption~\ref{cond:A}. Moreover, $\lim_{k\to\infty}\frac{\omega_k^2}{\alpha_k}=K(0)$.
	
	(d) There exists a constant $c>0$ such that for any sufficiently large $k$, we have that $\omega_k>1$ and, for any $q \in (0,1)$, we have
	\begin{equation} \label{ineq:ksin}
		\left|\alpha_k\frac{\Ksin(\omega)}{\omega}-1\right|\geq \frac{c}{\omega_k}\left|\omega-\omega_k\right|,
	\end{equation}
	 for all $\omega\in[\omega_k-\omega_k^q,\omega_k+\omega_k^q]$.
\end{lemma}

\begin{remark} \label{rem1} We will see later in the proof of Lemma~\ref{lem:complete-monotone:fourier-analysis} (d) that the assumption~$K(0)<\infty$ is crucial for our analysis, which explains why we are restricted to memory kernels that belong to $\CM_b$, not those belonging to $\CM\setminus\CM_b$, e.g., $|t|^{-\alpha}$, $\alpha>0$.
\end{remark}
\begin{proof}[Proof of Lemma~\ref{lem:complete-monotone:fourier-analysis}] (a) In view of Lemma~\ref{lem:complete-monotone:fourier}, the first assertion is evident since
\begin{align*}
\Kcos(\omega)=\int_0^\infty\frac{x}{x^2+\omega^2}\mu( \d x    ).
\end{align*}
To see the second claim, we employ Young's inequality to find that
\begin{align*}
\omega\Kcos(\omega)=\int_0^\infty\close\frac{x\omega}{x^2+\omega^2}\mu(\d x  )\leq \frac{1}{2}\int_0^\infty\close\mu(\d x  )<\infty.
\end{align*}
Also,
\begin{align*}
\omega^2\Kcos(\omega)=\int_0^\infty\close\frac{x\omega^2}{x^2+\omega^2}\mu(\d x  ),
\end{align*}
which is clearly increasing on $\omega\in(0,\infty)$.

(b) We note that
\begin{align*}
\omega\Ksin(\omega) = \int_0^\infty\close\frac{\omega^2}{x^2+\omega^2}\mu( \d x    ).
\end{align*}
The Monotone Convergence Theorem then implies that
\begin{align*}
\lim_{\omega\to\infty}\omega\Ksin(\omega)=\int_0^\infty\close \mu( \d x  )=K(0).
\end{align*}

(c) The first assertion is evident since $\Ksin(\omega)/\omega$ admits the formula
\begin{align*}
\frac{\Ksin(\omega)}{\omega}=\int_0^\infty \frac{1}{x^2+\omega^2}\mu( \d x  ).
\end{align*}
By the Monotone Convergence Theorem, we see that $\Ksin(\omega)/\omega\downarrow 0$ as $\omega\to\infty$. Now to verify that $1-\alpha_k\Ksin(\omega)/\omega=0$ must have a unique solution $\omega_k\in(0,\infty)$ for $k$ sufficiently large, we have the following observations
$$\lim_{\omega\to 0^+}1-\alpha_k\frac{\Ksin(\omega)}{\omega}=1-\alpha_k\int_0^\infty\frac{1}{x^2}\mu(\d x),\quad\text{and}\quad \lim_{\omega\to\infty}1-\alpha_k\frac{\Ksin(\omega)}{\omega}=1,$$
where the left--hand side limit may be positive. If $\int_0^\infty\frac{1}{x^2}\mu(\d x)$ diverges, then it is clear that
$$\lim_{\omega\to 0^+}1-\alpha_k\frac{\Ksin(\omega)}{\omega}=1-\alpha_k\int_0^\infty\frac{1}{x^2}\mu(\d x)=-\infty.$$
Otherwise, since $\{\alpha_k\}_{k\ge 1}\uparrow\infty$ by Assumption~\ref{cond:A}, there must exist an index $k^*$ sufficiently large such that for all $k\ge k^*$, the above left--hand side limit must be negative. Together with monotonic property, we can infer the existence and uniqueness of $\omega_k\in(0,\infty)$ solving $1-\alpha_k\frac{\Ksin(\omega)}{\omega}=0$. Finally, we note that
\begin{align*}
\frac{\omega_k^2}{\alpha_k}=\omega_k\Ksin(\omega_k)\to K(0), \qquad k\to\infty,
\end{align*}
by virtue of part (b). 

(d) Since $\{\alpha_k\}_{k\ge 1}\uparrow\infty$ and $\omega_k^2/\alpha_k\to K(0)$ from part (c), it is clear that for $k$ large enough, $\omega_k>1$. To prove~\eqref{ineq:ksin}, we assume that $\omega\neq\omega_k$ (it is trivial when $\omega=\omega_k$). Using the identity $\omega_k=\alpha_k\Ksin(\omega_k)$, we recast the left--hand side of~\eqref{ineq:ksin} as
\begin{align*}
\alpha_k\frac{\Ksin(\omega)}{\omega}-1&= \alpha_k\frac{\Ksin(\omega)}{\omega}-\alpha_k\frac{\Ksin(\omega_k)}{\omega_k}\\
&=\alpha_k\int_0^\infty\close\frac{1}{x^2+\omega^2}-\frac{1}{x^2+\omega_k^2}\mu( \d x    )\\
&=\alpha_k\int_0^\infty\close\frac{(\omega_k-\omega)(\omega_k+\omega)}{(x^2+\omega^2)(x^2+\omega_k^2)}\mu( \d x    ).
\end{align*}
It follows that~\eqref{ineq:ksin} is equivalent to
\begin{align*}
\alpha_k\int_0^\infty\close\frac{|\omega_k-\omega|(\omega_k+\omega)}{(x^2+\omega^2)(x^2+\omega_k^2)}\mu( \d x    )\geq \frac{c}{\omega_k}|\omega_k-\omega|,
\end{align*}
which is the same as
\begin{align*}
\int_0^\infty\close\frac{\omega_k(\omega_k+\omega)}{(x^2+\omega^2)(x^2+\omega_k^2)}\mu( \d x    )\geq \frac{c}{\alpha_k}=\frac{c\Ksin(\omega_k)}{\omega_k}=c\int_0^\infty\close\frac{1}{x^2+\omega_k^2}\mu( \d x    ).
\end{align*}
It therefore suffices to show that
\begin{align*}
\int_0^\infty\close\frac{\omega_k(\omega_k+\omega)}{(x^2+\omega^2)(x^2+\omega_k^2)}\mu(\d x  )\geq c\int_0^\infty\close\frac{1}{x^2+\omega_k^2}\mu(\d x  ).
\end{align*}
For $\omega\in[\omega_k-\omega_k^q,\omega_k+\omega_k^q]$, it is clear that $x^2+\omega^2\leq 4(x^2+\omega_k^2)$. We then find a lower estimate as follows:
\begin{align*}
\int_0^\infty\frac{\omega_k(\omega_k+\omega)}{(x^2+\omega^2)(x^2+\omega_k^2)}\mu(\d x  )&\geq\frac{1}{4}\int_0^\infty\close\frac{\omega_k^2}{(x^2+\omega_k^2)^2}\mu( \d x    )\\
&=\frac{\omega_k^2}{4K(0)}\int_0^\infty\close\mu( \d x    )\int_0^\infty\close\frac{1}{(x^2+\omega_k^2)^2}\mu( \d x ).
\end{align*}
We invoke H\"older's inequality (assuming $k$ is large enough) to find
\begin{equation*}
\begin{aligned}
\frac{\omega_k^2}{4K(0)}\int_0^\infty\close\mu( \d x    )\int_0^\infty\close\frac{1}{(x^2+\omega_k^2)^2}\mu( \d x  )&\geq\frac{\omega_k^2}{4K(0)}\Big(\int_0^\infty\close\frac{1}{x^2+\omega_k^2}\mu( \d x  )\Big)^2\\
&=\frac{\omega_k\Ksin(\omega_k)}{4K(0)}\int_0^\infty\close \frac{1}{x^2+\omega_k^2}\mu(\d x)\\
&\geq \frac{\Ksin(1)}{4K(0)}\int_0^\infty\close\frac{1}{x^2+\omega_k^2}\mu(\d x).
\end{aligned}
\end{equation*}
where in the last implication, we have employed the fact that $\omega\Ksin(\omega)$ is increasing from part (b). Finally, setting $c=\Ksin(1)/4K(0)$ and collecting everything, we obtain~\eqref{ineq:ksin}, which completes the proof.
\end{proof}

\section{Proof of main results} \label{sec:main-results}

\subsection{Proof of Theorem~\ref{thm:spatial-gle:well-posed}}
We begin this section by recalling some results on the well--posedness of the GLE \eqref{eq:spatial-gle:mode}. The proof of these results can be found in~\cite{mckinley2018anomalous}.

\begin{proposition} \label{prop:spatial-gle:mode:weak-solution} Suppose that the memory kernel $K$ satisfies Assumption~\ref{cond:K}. For each $k$, 
$V_k$ is a weak solution for \eqref{eq:spatial-gle:mode} in the sense of Definition~\ref{def:GLE:weak-solution} if and only if the spectral measure $\nu_k$ satisfies $\nu_k(\d\omega)=\rho_k(\omega)\d\omega$ where $\rho_k$ is given by
\begin{equation} \label{eq:spectral-density}
\rho_k(\omega) :=\frac{\lambda_k^2\widehat{K}(\omega)}{2\pi\left| i\omega+\alpha_k\widehat{K^+}(\omega)\right|^2}.
\end{equation}
\end{proposition}
\begin{proof}
See \cite[Theorem 4.3]{mckinley2018anomalous}.
\end{proof}

The function $\rho_k$ is called the \emph{spectral density} of $V_k$. Having obtained weak solutions of~\eqref{eq:spatial-gle:mode}, in the following proposition we assert that $u_k(t)$ is a stationary Gaussian process. 

\begin{proposition} \label{prop:spatial-gle:mode:process} Under the hypotheses of Proposition~\ref{prop:spatial-gle:mode:weak-solution}, let $V_k$ be the weak solution of~\eqref{eq:spatial-gle:mode} and $\rho_k$ be the corresponding spectral density as in~\eqref{eq:spectral-density}. Then, 
\begin{enumerate}[(a)]
\item $\rho_k\in L^1(\rbb)$;

\item $\delta_t $ belongs to $\dom(V_k)$ and the process $u_k(t):=\la V_k,\delta_t\ra$ is a real valued Gaussian stationary process with zero mean in the sense of Definition~\ref{def:stationary-process}, which is a.s.~continuous; and

\item the autocovariance of $u_k(t)$ admits the representation
\begin{align} \label{eq:spatial-gle:mode:covariance}
\E[u_k(t)u_k(s)] = \int_\rbb e^{i(t-s)\omega}\rho_k(\omega)\d\omega.
\end{align}
\end{enumerate}
\end{proposition}
\begin{proof}
See \cite[Theorem 5.4]{mckinley2018anomalous}.
\end{proof}

With regard to the existence of stationary solutions for~\eqref{eq:spatial-gle}, we remark that Proposition~\ref{prop:spatial-gle:mode:process} only guarantees the second condition of Definition~\ref{def:spatial-gle:solution}. In order to fulfill the first requirement, using the decomposition $u(t,\xbf)=\sum_{k\ge 1}u_k(t)e_k(\xbf)$ and the fact that the processes $u_k(t)$ are all independent, we observe that
$$\E\|u(t,\cdot)\|^2_H=\sum_{k\ge 1}\E|u_k(t)|^2.$$
It is thus necessary to obtain useful bounds on the second moment of $u_k$, which by virtue of Proposition~\ref{prop:spatial-gle:mode:process} (c), cf.~\eqref{eq:spatial-gle:mode:covariance}, is equivalent to calculating $\int_\rbb \rho_k(\omega)\d\omega$. To be more precise, we have the following lemma, which will be employed to prove Theorem~\ref{thm:spatial-gle:well-posed}.

\begin{lemma} \label{lem:int.rho_k=lamda/alpha}  Suppose that the memory kernel $K$ satisfies Assumption~\ref{cond:K}. Let $\rho_k$ be the spectral density as in~\eqref{eq:spectral-density}. Then,
\begin{equation} \label{eq:int.rho_k=lambda/alpha}
\int_\rbb \rho_k(\omega)\d\omega=\frac{\lambda_k^2}{\alpha_k}.
\end{equation}
\end{lemma}
We note that Lemma~\ref{lem:int.rho_k=lamda/alpha} may be considered as an extension of previous results on the second moment of the solutions of 1D GLE, e.g.,~\cite[Equation (2.7)]{hohenegger2017equipartition} and \cite[Proposition 2.1]{hohenegger2017fluid}, in which the kernels considered have a special finite-sum-of-exponentials form. By employing a contour integral method similar to what is used in ~\cite{hohenegger2017equipartition,kou2008stochastic}, we are able to generalize these results to any kernel satisfying Assumption~\ref{cond:K}. The proof of Lemma~\ref{lem:int.rho_k=lamda/alpha} will be presented in detail at the end of the section. With Lemma~\ref{lem:int.rho_k=lamda/alpha} in hand, we are ready to assert the existence of weak stationary solutions of~\eqref{eq:spatial-gle}. Since the argument is short, we include it here for the sake of completeness.

\begin{proof}[Proof of Theorem~\ref{thm:spatial-gle:well-posed}]
Let $u_k$ be the stationary Gaussian process defined in Proposition~\ref{prop:spatial-gle:mode:process}. Verifying the existence of $u(t,\xbf)=\sum_{k\ge 1}u_k(t)e_k(\xbf)$ essentially amounts to checking the first condition of Definition~\ref{def:spatial-gle:solution}. To this end, in view of~\eqref{eq:spatial-gle:mode:covariance} together with~\eqref{eq:int.rho_k=lambda/alpha}, we have 
\begin{align*}
\Enone\|u(t,\cdot)\|_H^2=\sum_{k\geq 1}\Enone|u_k(t)|^2=\sum_{k\ge 1}\int_0^\infty\close  \rho_k(\omega)\d\omega=\sum_{k\ge 1}\frac{\lambda_k^2}{\alpha_k}.
\end{align*}
This implies that Definition~\ref{def:spatial-gle:solution} (a) ($\Enone\|u(t,\cdot)\|_H^2$ being finite) is equivalent to Assumption~\ref{cond:alpha-k:wellposed} being satisfied. The proof is thus complete.
\end{proof}

We now discuss the proof of Lemma~\ref{lem:int.rho_k=lamda/alpha}. As mentioned earlier, we will follow closely the strategy in~\cite{didier2021generalized} and make use of contour integrals on the complex plane $\cbb$ to evaluate $\int_\rbb \rho_k(\omega)\d\omega$. We note that in previous results \cite{hohenegger2017equipartition,hohenegger2017fluid} for memory kernels having a sum-of-exponentials form, the typical approach \cite{hohenegger2017equipartition,kou2008stochastic} is to employ functions similar to $f_k(z)$ as in~\eqref{form:f_k(z)} below and their contour integrals on the upper half complex plane. The arguments for these results rely on a careful analysis on the locations of the poles for the functions therein. In our method, instead of working with the upper half complex plane, we shift the analysis to the lower half plane. The novelty in this approach is that the function that we study, Eq.~\eqref{form:f_k(z)}, is actually analytic, which allows us to include more general kernels, such as those described in Assumption~\ref{cond:K}. To be precise, we have the following lemma that will be employed to prove Lemma~\ref{lem:int.rho_k=lamda/alpha}. 

\begin{lemma} \label{lem:f_k} Suppose that the memory kernel $K$ satisfies Assumption~\ref{cond:K}. Let $f_k(z)$ be a complex--valued function given by
\begin{equation} \label{form:f_k(z)}
f_k(z)=\frac{1}{i\big(z-\alpha_k\Ksin(z)\big)+\alpha_k\Kcos(z)},
\end{equation}
where $\alpha_k$ is as defined in Assumption~\ref{cond:A}. Then, $f_k(z)$ is analytic on the lower half complex plane $\cbb^-\sm$ where
$$\cbb^-=\{z\in\cbb:\emph{\Im}(z)\le 0\}.$$
\end{lemma}
\begin{proof} First of all, in view of~\eqref{eq:complete-monotone:fourier} and~\eqref{form:KcosKsin}, we recast $f_k(z)$ as
\begin{align} \label{form:f_k:z=u-iv}
f_k(z)&=\frac{1}{\alpha_k (\Kcos(z)-i\Ksin(z))+iz}=\frac{1}{\alpha_k \int_0^\infty \frac{1}{iz+x}\mu(\d x)+iz}.
\end{align}
We now proceed to prove that $(f_k(z))^{-1}$ given by
$$(f_k(z))^{-1}=\alpha_k \int_0^\infty\close \frac{1}{iz+x}\mu(\d x)+iz,$$
is analytic and does not admit any root in $\cbb^-\sm$, which in turn implies that $f_k(z)$ is analytic in $\cbb^-\sm$.

To verify the analyticity of $(f_k(z))^{-1}$, it suffices to show $\int_0^\infty \frac{1}{iz+x}\mu(\d x)$ is analytic in $\cbb^-\sm$. To this end, for any $z_0\in\cbb^-\sm$, consider $z\in \cbb$ such that $|z-z_0|<|z_0|/2$ and observe that
\begin{align*}
\int_0^\infty\close \frac{1}{iz+x}\mu(\d x)&=\int_0^\infty\close \frac{1}{(iz_0+x)\big(\frac{iz-iz_0}{iz_0+x}+1\big)}\mu(\d x)=\sum_{n\ge 0}\int_0^\infty\close \frac{(-i)^n}{(iz_0+x)^{n+1}}\mu(\d x)(z-z_0)^n.
\end{align*}
We note that the last implication above is still formal. We now claim that for any $z$ such that $|z-z_0|<|z_0|/2 $, the right hand side series above actually converges absolutely. Indeed, by using the fact that $z_0=u-iv\in\cbb^-\sm$ with $u\in\rbb$, $v\ge 0$ and $u^2+v^2\neq 0$, we have the estimate
\begin{align*}
\sum_{n\ge 0}\int_0^\infty\close \frac{1}{|iz_0+x|^{n+1}}\mu(\d x)|z-z_0|^n&\le \sum_{n\ge 0}\int_0^\infty\close \frac{|z_0|^n}{2^n|iz_0+x|^{n+1}}\mu(\d x)\\
&=\sum_{n\ge 0}\int_0^\infty\close \frac{(u^2+v^2)^{n/2}}{2^n|u^2+(x+v)^2|^{(n+1)/2}}\mu(\d x)\\
&\le \sum_{n\ge 0}\int_0^\infty\close \frac{1}{2^n|u^2+v^2|^{1/2}}\mu(\d x)\\
&=\frac{1}{|z_0|}\mu([0,\infty))<\infty,
\end{align*}
where the last implication follows from the fact that $\mu$ is a finite measure as $K\in \CMb$, cf. Remark~\ref{rem:CMb}. This proves the analyticity of $(f_k(z))^{-1}$.

To verify that $(f_k(z))^{-1}$ does not have any roots in $\cbb^-\sm$, similar to the above estimates, we rewrite $z=u-iv$ where $u\in\rbb$, $v\ge 0$, $u^2+v^2\neq 0$, and observe that after a tedious but routine calculation
\begin{align*} 
\Re\big((f_k(z))^{-1}\big)=\alpha_k \int_0^\infty\close  \frac{x+v}{(x+v)^2+u^2}\mu(\d x)+v>0,
\end{align*}
since $\mu$ is not null on $[0,\infty)$. The proof is thus complete.

\end{proof}

We now give the proof of Lemma~\ref{lem:int.rho_k=lamda/alpha}, which is a slightly rework of the proof of \cite[Lemma 4.4]{didier2021generalized} tailored to our setting. See also \cite[Theorem 4.2]{kou2008stochastic}.
\begin{proof}[Proof of Lemma~\ref{lem:int.rho_k=lamda/alpha}] We first note that the spectral density $\rho_k$ in~\eqref{eq:spectral-density} can be written as 
\begin{align} \label{eq:spectral-density:transform}
\rho_k(\omega)= \frac{1}{\pi}\cdot\frac{\lambda_k^2\Kcos(\omega)}{\alpha^2_k\Kcos(\omega)^2+\big(\omega-\alpha_k\Ksin(\omega)\big)^2},
\end{align}
since $K$ is assumed to be even, and thus, $\widehat{K}(\omega)=2\Kcos(\omega)$. It follows that 
$$\int_\rbb\rho_k(\omega)\d\omega=\frac{1}{\pi}\int_0^\infty\close \frac{2\lambda_k^2\Kcos(\omega)}{\alpha^2_k\Kcos(\omega)^2+\big(\omega-\alpha_k\Ksin(\omega)\big)^2}\d\omega.$$
We aim to make use of contour integrals of $f_k(z)$ as in~\eqref{form:f_k(z)} to calculate the above integral. For $R>0$, we introduce the outer and inner half circles in $\cbb^-\sm$ given by
\begin{equation} \label{form:curve:C_R^-.and.C_(1/R)^-}
 C_R=\{Re^{i\theta}:-\pi\le \theta\le 0\} \quad\text{and}\quad C_{1/R}= \{e^{i\theta}/R:-\pi\le \theta\le 0\}.
\end{equation}
Also, let $C(R)$ denote the closed curve in $\cbb^-\sm$ oriented clockwise as follows:
\begin{equation}\label{form:curve:C(R)}
C(R)=[-R,-1/R]\cup C_{1/R}\cup[1/R,R]\cup C_R.
\end{equation}
Recall $f_k(z)$ from~\eqref{form:f_k(z)}. In light of Lemma~\ref{lem:f_k}, $f_k$ is analytic in $\cbb^-\sm$, implying that for all $R>0$
\begin{align*}
\int_{C(R)}\close f_k(z)\d z=0.
\end{align*}
On the other hand, we can decompose the above contour integral as follows:
\begin{align*}
\int_{C(R)}\close f_k(z)\d z&=\Big\{\int_{-R}^{-1/R}\close +\int_{C_{1/R}}\close +\int_{1/R}^R+\int_{C_R}\Big\}f_k(z)\d z\\&=I_1(R)+I_2(R)+I_3(R)+I_4(R).
\end{align*}

In view of the expression for $f_k$ in~\eqref{form:f_k(z)}, we have
$$I_3(R) =\int_{1/R}^R\frac{\d\omega}{\alpha_k\Kcos(\omega) + i(\omega-\alpha_k \Ksin(\omega))}.$$

Concerning $I_1(R)$, we recall that $\Kcos(\omega)$ is even whereas $\Ksin(\omega)$ is odd. Thus, by a change of variable $z:=-\omega$, we obtain
\begin{align*}
I_1(R)
&=\int_{1/R}^{R} \frac{\d \omega}{\alpha_k \Kcos(\omega) - i(\omega-\alpha_k \Ksin(\omega))}.
\end{align*}
It follows immediately that
$$I_1(R) +I_3(R) =\int_{1/R}^R\frac{2\alpha_k\Kcos(\omega)}{\alpha_k^2 \Kcos(\omega)^2 +\big(\omega-\alpha_k \Ksin(\omega)\big)^2}\d\omega.$$
Since $\Kcos(\omega)>0$, cf.~\eqref{eq:complete-monotone:fourier}, by the Monotone Convergence Theorem, we obtain
$$I_1(R) +I_3(R)\to\pi \frac{\alpha_k}{\lambda_k^2}\int_0^\infty\close \rho_k(\omega)\d\omega\quad\text{as}\quad R\to\infty,$$
where $\rho_k$ is as in~\eqref{eq:spectral-density:transform}.

Concerning $I_2(R)$ on the inner half circle $C_{1/R}$, we aim to show that its limit is zero as $R$ tends to infinity. Indeed, recall the form of $f_k(z)$ given in~\eqref{form:f_k:z=u-iv}. For $z=u-iv\in \cbb^-\sm$ such that $|z|$ is small, namely 
\begin{displaymath}
	|z|<\min\Big\{1,\frac{\alpha_k}{8\sqrt{2}}\int_0^\infty \frac{1}{x+1}\mu(\d x)\Big\},
\end{displaymath}
observe that
\begin{align*}
|f_k(z)^{-1}|&=\Big|\alpha_k\int_0^\infty\close  \frac{x+v}{(x+v)^2+u^2}\mu(\d x)+v+i\Big(u-\alpha_k\int_0^\infty \close \frac{u}{(x+v)^2+u^2}\mu(\d x)\Big)\Big|\\
&\ge \alpha_k\Big|\int_0^\infty\close  \frac{x+v}{(x+v)^2+u^2}\mu(\d x)-i\int_0^\infty \close \frac{u}{(x+v)^2+u^2}\mu(\d x)\Big)\Big|-|v+iu|\\
&\ge \frac{\alpha_k}{\sqrt{2}}\int_0^\infty\close  \frac{x+v+|u|}{(x+v)^2+u^2}\mu(\d x)-|z|.
\end{align*}
To obtain the last inequality, we employed the lower bound $\sqrt{2}|a-ib|\ge |a|+|b|$ for $a,b\in\rbb$. To further bound $f_k(z)^{-1}$ from below, we use the elementary inequality for $x, v\ge 0$ and $|v|, |u|\le 1$
$$\frac{x+v+|u|}{(x+v)^2+u^2}\ge \frac{1}{2(x+1)},$$
to estimate
\begin{align*}
\frac{\alpha_k}{\sqrt{2}}\int_0^\infty\close  \frac{x+v+|u|}{(x+v)^2+u^2}\mu(\d x)-|z|
&\ge \frac{\alpha_k}{2\sqrt{2}}\int_0^\infty\close \frac{1}{x+1}\mu(\d x)-|z|\\
&\ge \frac{\alpha_k}{4\sqrt{2}}\int_0^\infty\close \frac{1}{x+1}\mu(\d x)>0,
\end{align*}
since $\mu$ is not null on $[0,\infty)$. As a consequence, we obtain
\begin{align*}
|f_k(z)^{-1}|\ge  \frac{\alpha_k}{4\sqrt{2}}\int_0^\infty\close \frac{1}{x+1}\mu(\d x).
\end{align*}
Thus, by making the change of variable $z:=R^{-1}e^{i\theta}$ for all $R$ sufficiently large, we have
\begin{align*}
|I_2(R)|&=\Big|\int_{-\pi}^0\frac{R^{-1}e^{i\theta}i\d\theta}{\alpha_k \Kcos(R^{-1} e^{i\theta}) + i(R^{-1} e^{i\theta}-\alpha_k \Ksin(R^{-1} e^{i\theta}))}\Big|\\
&\le\frac{4\sqrt{2}R^{-1}}{\alpha_k\int_0^\infty\frac{1}{x+1}\mu(\d x)} \int_{-\pi}^0 \d\theta,
\end{align*}
which clearly converges to zero as $R$ tends to infinity. 

Likewise, for $I_4(R)$, we note that for $z=u-iv\in \cbb^-\sm$ such that $|z|>1$,
\begin{align*}
|\Kcos(z)-i\Ksin(z)|&=\Big|\int_0^\infty \frac{x+v-iu}{(x+v)^2+u^2}\mu(\d x) \Big|\\
&\le \int_0^\infty\close  \frac{x+v+|u|}{(x+v)^2+u^2}\mu(\d x)\\
&\le \int_0^\infty\close  \frac{2}{x+v+|u|}\mu(\d x)\\
&\le 2\int_0^\infty\close  \mu(\d x)=2\mu([0,\infty)),
\end{align*}
since $\mu$ is assumed to be finite measure on $[0,\infty)$, cf. Remark~\ref{rem:CMb}. It follows that by making the change of variable $z:=Re^{i\theta}$, it holds that
\begin{align*}
I_4(R)&=\int_0^{-\pi}\close \frac{Re^{i\theta}i\d\theta}{\alpha_k \Kcos(R e^{i\theta}) +i(R e^{i\theta}-\alpha_k \Ksin(R e^{i\theta}))}\\
&=\int_0^{-\pi}\close \frac{\d\theta}{\frac{\alpha_k }{iRe^{i\theta}}  \big(\Kcos(Re^{i\theta})-i\Ksin(Re^{i\theta})\big)  +1 },
\end{align*}
which converges to $-\pi$ as $R$ tends to infinity by virtue of the Dominated Convergence Theorem. 

We collect the above limits to arrive at the the following
\begin{align*}
0=\lim_{R\to\infty}\int_{C(R)}\close f_k(z)\d z = \pi\frac{\alpha_k}{\lambda_k^2}\int_0^\infty \close \rho_k(\omega)\d\omega-\pi,
\end{align*}
which in turn implies~\eqref{eq:int.rho_k=lambda/alpha}. This finishes the proof.

\end{proof}

\subsection{Proof of Theorem~\ref{thm:spatial-gle:Holder}}

We now turn our attention to the main theorem of paper, namely, the regularity of weak stationary solutions $u(t,\xbf)$. As discussed in Section~\ref{sec:result}, in order to prove Theorem~\ref{thm:spatial-gle:Holder}, we will employ the classical Kolmogorov's criterion to establish H\"older continuity. To this end, we must obtain useful estimates on differences in time, $u_k(t)-u_k(s)$ ($t,s\in\rbb$), as well as their second moments under Assumption~\ref{cond:K}. In Proposition~\ref{prop:spatial-gle:mode:Holder} below, we assert a bound on the difference $u_k(t)-u_k(s)$, which will be employed later to prove H\"older regularities in time of $u(t,\xbf)$.
\begin{proposition}  \label{prop:spatial-gle:mode:Holder}
Suppose that the memory kernel $K$ satisfies Assumption~\ref{cond:K}. For $k\geq 1$, let $u_k(t)$ be the stationary process as in Proposition~\ref{prop:spatial-gle:mode:process}. Then there exists an index $k^*$ large enough such that for any~$\alpha,\,q\in(0,1)$, there exists a constant $c=c(\alpha,q,k^*)>0$ independent of all $k$ such that for all $t,\,s\in\rbb$, $1\le k\le k^*$,
\begin{equation} \label{ineq:spatial-gle:mode:Holder:k<k^*}
\E|u_k(t)-u_k(s)|^2<c|t-s|^\alpha,
\end{equation}
and for all $k>k^*$,
\begin{equation} \label{ineq:spatial-gle:mode:Holder}
\E|u_k(t)-u_k(s)|^2<c \frac{\lambda_k^2}{\alpha_k^{q-\alpha/2}}|t-s|^\alpha.
\end{equation}
\end{proposition}
As a consequence of Proposition~\ref{prop:spatial-gle:mode:Holder}, we have the following corollary asserting a useful bound on the second moment of $u_k(t)$, which will be employed to prove the spatial H\"older regularity of $u(t,\xbf)$.
\begin{corollary} \label{cor:Holder_space}
Suppose that the memory kernel $K$ satisfies Assumption~\ref{cond:K}. For $k\geq 1$, let $u_k(t)$ be the stationary process as in Proposition~\ref{prop:spatial-gle:mode:process}. Let $k^*$ be the same index as in Proposition~\ref{prop:spatial-gle:mode:Holder}. Then for any~$q\in(0,1)$, there exists a constant $c=c(q,k^*)>0$ such that for all $t\in\rbb$, $1\le k\le k^*$,
\begin{equation} \label{ineq:spatial-gle:mode:Holder:2:k<k^*}
\mathbb{E}|u_k(t)|^2< c,
\end{equation}
and for all $k>k^*$, 
\begin{equation} \label{ineq:spatial-gle:mode:Holder:2}
\mathbb{E}|u_k(t)|^2< c\frac{\lambda_k^2}{\alpha_k^{q}}.
\end{equation}
\end{corollary}

The proofs of Proposition~\ref{prop:spatial-gle:mode:Holder} and Corollary~\ref{cor:Holder_space} will be deferred to the end of this section. We are now in a position to prove Theorem~\ref{thm:spatial-gle:Holder}.

\begin{proof}[Proof of Theorem~\ref{thm:spatial-gle:Holder}] Let $\eta$ be the constant from Assumption~\ref{cond:alpha-k:Holder}. Since the processes $u_k(\cdot)$ are mutually independent with zero mean, in view of Proposition~\ref{prop:spatial-gle:mode:Holder}, cf.~\eqref{ineq:spatial-gle:mode:Holder:k<k^*} and~\eqref{ineq:spatial-gle:mode:Holder}, we have the following estimates for any $t,\,s\in\rbb$ and $\xbf\in\domain$
\begin{align*}
\Enone|u(t,\xbf)-u(s,\xbf)|^2 &= \sum_{k\geq 1}\Enone|u_k(t)-u_k(s)|^2|e_k(\xbf)|^2 \\
&\leq c|t-s|^{\alpha}\sum_{k\geq 1}\frac{\lambda_k^2c_k^2}{\alpha_k^{q-\alpha/2}}\\
&< c|t-s|^{\alpha}\sum_{k\geq 1}\frac{\lambda_k^2c_k^2}{\alpha_k^{\eta}} ,
\end{align*}
where the last implication follows from the choice of $q,\,\alpha\in (0,1)$ satisfying $\eta<q-\alpha/2$, which in turn is always possible for any $\alpha/2\in(0,1-\eta)$. Since $u(t,\xbf)-u(s,\xbf)$ is Gaussian, we infer a constant $C(m)>0$ for $m>0$ such that
\begin{align*}
\Enone|u(t,\xbf)-u(s,\xbf)|^{2m} \leq C(m)|t-s|^{m\alpha}.
\end{align*}
By Kolmogorov's test for stochastic processes \cite[Theorem 3.3]{da2014stochastic}, there exists a version of $u$ that is H\"older continuous in $t$ for every $\gamma\in\big(0,\frac{\alpha}{2}-\frac{1}{2m}\big)$. By choosing $m$ sufficiently large and $\alpha/2$ as close to $1-\eta$ as possible, we obtain H\"older continuity in time $t$ for every $\gamma\in(0,1-\eta)$.

With regard to spatial regularity, we note that Assumption~\ref{cond:A} on $\nabla e_k$ implies the following bound for $\xbf,\ybf\in\domain$ and $\alpha\in(0,1)$ (see  \cite[Lemma 5.21]{da2014stochastic}) 
\begin{align*}
|e_k(\xbf)-e_k(\ybf)| \leq c(\alpha) \alpha_k^{\alpha/2}c_k|\xbf-\ybf|^\alpha.
\end{align*}
We then apply Corollary~\ref{cor:Holder_space}, cf.~\eqref{ineq:spatial-gle:mode:Holder:2:k<k^*} and~\eqref{ineq:spatial-gle:mode:Holder:2}, for $q,\,\alpha\in(0,1)$ to find
\begin{align*}
\Enone|u(t,\xbf)-u(t,\ybf)|^2 &\leq \sum_{k\geq 1}\Enone|u_k(t)|^2|e_k(\xbf)-e_k(\ybf)|^2 \\
&\leq c|\xbf-\ybf|^{2\alpha}\sum_{k\geq 1}\frac{\lambda_k^2 c_k^2}{\alpha_k^{q-\alpha}}\\
&< c|\xbf-\ybf|^{2\alpha}\sum_{k\geq 1}\frac{\lambda_k^2 c_k^2}{\alpha_k^{\eta}},
\end{align*}
which again is always possible for $\alpha\in(0,1-\eta)$. We then arrive at the estimate 
\begin{align*}
\Enone|u(t,\xbf)-u(t,\ybf)|^{2m} \leq C(m)|\xbf-\ybf|^{2m\alpha}.
\end{align*}
By the Kolmogorov test for random fields \cite[Theorem 3.5]{da2014stochastic}, $u$ is H\"older continuous in space for any $\gamma>0$ up to $\frac{2m\alpha-d}{2m}=\alpha-\frac{d}{2m}$ where $d$ is the spatial dimension. We finally choose $m$ sufficiently large and $\alpha$ as close to $1-\eta$ as possible to obtain $\gamma-$H\"older continuity in space for any $\gamma\in(0,1-\eta)$. This finises the proof.
\end{proof}

We now turn to the proof of Proposition~\ref{prop:spatial-gle:mode:Holder}. In order to establish Proposition~\ref{prop:spatial-gle:mode:Holder}, we will make use of the following elementary inequality: for any $\alpha\in(0,1)$ and $x\in\rbb$, we have
\begin{align}\label{lem:trig} 
1-\cos(x)\le\frac{2}{\alpha}|x|^\alpha.
\end{align}
Indeed, observe that if $|x|\ge 1$, then since $\alpha\in(0,1)$
\begin{align*}
\frac{2}{\alpha}|x|^\alpha\ge 2\ge 1-\cos(x).
\end{align*}
On the other hand, if $|x|<1$, then we have
\begin{align*}
1-\cos(x)\le \frac{x^2}{2}\le \frac{2}{\alpha}|x|^\alpha.
\end{align*}
We are now in a position to prove Proposition~\ref{prop:spatial-gle:mode:Holder}.
\begin{proof}[Proof of Proposition~\ref{prop:spatial-gle:mode:Holder}]
In view of~\eqref{eq:spatial-gle:mode:covariance} and~\eqref{eq:spectral-density:transform}, a straightforward calculation shows that
\begin{align*}
\Enone|u_k(t)-u_k(s)|^2&=\int_\rbb\close\big(2-e^{i(t-s)\omega}-e^{i(s-t)\omega}\big)\rho_k(\omega)\d\omega \\
&=\frac{2}{\pi}\int_0^\infty\close\close \left(1-\cos(\omega(t-s))\right)\frac{\lambda_k^2\Kcos(\omega)}{\alpha_k^2\Kcos^2(\omega)+\left(\omega-\alpha_k\Ksin(\omega)\right)^2}\d\omega\\
&\le\frac{4\lambda_k^2}{\pi\alpha}|t-s|^\alpha\int_0^\infty\close \frac{\omega^\alpha\Kcos(\omega)}{\alpha_k^2\Kcos^2(\omega)+\left(\omega-\alpha_k\Ksin(\omega)\right)^2}\d\omega,
\end{align*}
where in the last implication, for $\alpha\in(0,1)$, we have invoked Inequality~\eqref{lem:trig}. 

To verify~\eqref{ineq:spatial-gle:mode:Holder:k<k^*}, it suffices to prove that for any $k$, the above integral is finite. In view of Lemma~\ref{lem:complete-monotone:fourier}, we see that $$\lim_{\omega\to\infty}\Kcos(\omega)=\lim_{\omega\to\infty}\Ksin(\omega)=0,$$
implying the integrand is dominated by $\omega^{\alpha-2}$, which is integrable at infinity. On the other hand, when $\omega\to 0$, by Fatou's Lemma, it holds that $$\liminf_{\omega\to 0}\Kcos(\omega)\ge \int_0^\infty \frac{1}{x}\mu(\d x)>0,$$
implying the integrand is dominated by $\omega^\alpha(\alpha_k\int_0^\infty\frac{1}{x}\mu(\d x))^{-1}$ which is integrable around the origin. We thus combine two cases to infer the existence of a positive constant $c(k)$ such that
$$\int_0^\infty\close \frac{\omega^\alpha\Kcos(\omega)}{\alpha_k^2\Kcos^2(\omega)+\left(\omega-\alpha_k\Ksin(\omega)\right)^2}\,\d\omega\le c(k),$$
which proves~\eqref{ineq:spatial-gle:mode:Holder:k<k^*}.

Now, let $k^*$ be a large constant such that for $k>k^*$, $\omega_k>1$ is the unique solution on $(0,\infty)$ of Equation~\eqref{eq:complete-monotone:fourier:ksin} from Lemma~\ref{lem:complete-monotone:fourier-analysis} (c) and~(d). We now decompose the last integral as follows:
\begin{align*}
&\int_0^\infty\close  \frac{\omega^\alpha\Kcos(\omega)}{\alpha_k^2\Kcos^2(\omega)+\left(\omega-\alpha_k\Ksin(\omega)\right)^2}\d\omega\\
&=\Big\{\int_0^{\omega_k-\omega_k^q}+\int_{\omega_k-\omega_k^q}^{\omega_k-1}+\int_{\omega_k-1}^{\omega_k+1}+\int_{\omega_k+1}^{\omega_k+\omega_k^q}+\int_{\omega_k+\omega_k^q}^\infty\Big\} \frac{\omega^\alpha\Kcos(\omega)}{\alpha_k^2\Kcos^2(\omega)+\left(\omega-\alpha_k\Ksin(\omega)\right)^2}\d\omega\\
&= I_1+\dots+I_5.
\end{align*}
 
To estimate $I_1$, we recall from Lemma~\ref{lem:complete-monotone:fourier-analysis} (b) that $\Ksin(\omega)/\omega$ is decreasing on $\omega\in(0,\infty)$. Thus, for any $\omega\in(0,\omega_k-\omega_k^q]$, it follows that
\begin{align*}
\alpha_k\frac{\Ksin(\omega)}{\omega}-1\geq \alpha_k\frac{\Ksin(\omega_k-\omega_k^q)}{\omega_k-\omega_k^q}-1.
\end{align*} 
We then substitute $\omega:=\omega_k-\omega_k^q$ in Inequality~\eqref{ineq:ksin} to obtain
\begin{align*}
 \alpha_k\frac{\Ksin(\omega_k-\omega_k^q)}{\omega_k-\omega_k^q}-1\geq \frac{c}{\omega_k^{1-q}}.
\end{align*}
It follows that
\begin{align*}
I_1&=\int_0^{\omega_k-\omega_k^q}\close\close \frac{\omega^\alpha\Kcos(\omega)}{\alpha_k^2\Kcos^2(\omega)+\left(\omega-\alpha_k\Ksin(\omega)\right)^2}\d\omega\\
 &=\int_0^{\omega_k-\omega_k^q}\close\close \frac{\omega^\alpha\Kcos(\omega)}{\alpha_k^2\Kcos^2(\omega)+\omega^2\left(\alpha_k\frac{\Ksin(\omega)}{\omega}-1\right)^2}\d\omega\\
& \leq \int_0^{\omega_k-\omega_k^q}\close\close \frac{\omega^\alpha\Kcos(\omega)}{\alpha_k^2\Kcos^2(\omega)+c\omega^2/\omega_k^{2-2q} }\d\omega.
\end{align*}
We apply Young's product inequality to the above denominator to infer
\begin{align*}
 \int_0^{\omega_k-\omega_k^q}\close\close \frac{\omega^\alpha\Kcos(\omega)}{\alpha_k^2\Kcos^2(\omega)+c\omega^2/\omega_k^{2-2q} }\d\omega&\leq  c \int_0^{\omega_k-\omega_k^q} \frac{\omega_k^{1-q}\omega^{\alpha-1}}{\alpha_k }\d\omega
 \leq c\frac{\omega_k^{1-q+\alpha}}{\alpha_k}
 \leq \frac{c}{\alpha_k^{(1+q-\alpha)/2}},
\end{align*}
where in the last implication we have employed the fact that $\omega_k^2/\alpha_k$ is bounded uniformly with respect to $k$ since $\omega_k^2/\alpha_k\to K(0)$ as $k\to\infty$, by virtue of Lemma~\ref{lem:complete-monotone:fourier-analysis}~(c).

Similar to the argument on $I_1$, to estimate $I_5$, we note that if $\omega\in[\omega_k+\omega_k^q,\infty)$ then
\begin{align*}
1-\alpha_k\frac{\Ksin(\omega)}{\omega}\geq 1- \alpha_k\frac{\Ksin(\omega_k+\omega_k^q)}{\omega_k+\omega_k^q},
\end{align*}
and that substituting $\omega:=\omega_k+\omega_k^q$ in~\eqref{ineq:ksin} yields
\begin{align*}
1-\alpha_k\frac{\Ksin(\omega_k+\omega_k^q)}{\omega_k+\omega_k^q}\geq \frac{c}{\omega_k^{1-q}}.
\end{align*}
We then have a chain of implications
\begin{equation*}
\begin{aligned}
I_5=\int_{\omega_k+\omega_k^q}^\infty\frac{\omega^\alpha\Kcos(\omega)}{\alpha_k^2\Kcos^2(\omega)+\omega^2\left(\alpha_k\frac{\Ksin(\omega)}{\omega}-1\right)^2}\d\omega&\leq \int_{\omega_k+\omega_k^q}^\infty\frac{\omega^{\alpha}\Kcos(\omega)}{\omega^2\left(1-\alpha_k\frac{\Ksin(\omega_k+\omega_k^q)}{\omega_k+\omega_k^q}\right)^2} \d\omega\\
 &\leq c\int_{\omega_k+\omega_k^q}^\infty\frac{\omega^{\alpha}\Kcos(\omega)\omega_k^{2-2q}}{\omega^2} \d\omega \\
&=c \int_{\omega_k+\omega_k^q}^\infty\frac{\omega\Kcos(\omega)\omega_k^{2-2q}}{c\omega^{3-\alpha}}\d\omega\\
&\leq c\int_{\omega_k+\omega_k^q}^\infty\frac{\omega_k^{2-2q}}{\omega^{3-\alpha}}\d \omega,
\end{aligned}
\end{equation*}
where the last inequality follows from the fact that $\omega\Kcos(\omega)$ is bounded, by virtue of Lemma~\ref{lem:complete-monotone:fourier-analysis}~(a). We now integrate the above integral with respect to $\omega$ to find 
\begin{equation*}
 c\int_{\omega_k+\omega_k^q}^\infty\frac{\omega_k^{2-2q}}{\omega^{3-\alpha}}\\d\omega \leq c\frac{\omega_k^{2-2q}}{\omega_k^{2-\alpha}}\leq \frac{c}{\alpha_k^{q-\alpha/2}},
\end{equation*}
which implies that
\begin{align*}
I_5\leq \frac{c}{\alpha_k^{q-\alpha/2}}.
\end{align*}

Regarding $I_2$, we invoke Inequality~\eqref{ineq:ksin} again to find
\begin{equation*}
\begin{aligned}
I_2&=\int_{\omega_k-\omega_k^q}^{\omega_k-1}\frac{\omega^\alpha\Kcos(\omega)}{\alpha_k^2\Kcos^2(\omega)+\omega^2\left(\alpha_k\frac{\Ksin(\omega)}{\omega}-1\right)^2}\d\omega\\
&\leq \int_{\omega_k-\omega_k^q}^{\omega_k-1}\frac{\omega^{\alpha}\Kcos(\omega)}{\alpha_k^2\Kcos^2(\omega)+c\frac{\omega^2}{\omega_k^2}(\omega_k-\omega)^2}\d\omega.
\end{aligned}
\end{equation*}
Also, since $q\in(0,1)$ and $\omega_k\to\infty$ as $k\to\infty$ (Lemma~\ref{lem:complete-monotone:fourier-analysis}~(c)), for $k$ sufficiently large and any $\omega\in[\omega_k-\omega_k^q,\omega_k-1]$, the ratio $\omega/\omega_k$ is bounded from below uniformly in $k$. We then infer that
\begin{equation*}
\int_{\omega_k-\omega_k^q}^{\omega_k-1}\frac{\omega^{\alpha}\Kcos(\omega)}{\alpha_k^2\Kcos^2(\omega)+c\frac{\omega^2}{\omega_k^2}(\omega_k-\omega)^2}\d\omega 
\leq \int_{\omega_k-\omega_k^q}^{\omega_k-1}\frac{\omega_k^{\alpha}\Kcos(\omega)}{\alpha_k^2\Kcos^2(\omega)+c(\omega_k-\omega)^2}\d\omega.
\end{equation*}
Young's inequality now implies
\begin{equation*}
\begin{aligned}
\int_{\omega_k-\omega_k^q}^{\omega_k-1}\frac{\omega_k^{\alpha}\Kcos(\omega)}{\alpha_k^2\Kcos^2(\omega)+c(\omega_k-\omega)^2}\d\omega &\leq \frac{c\omega_k^\alpha}{\alpha_k} \int_{\omega_k-\omega_k^q}^{\omega_k-1}\frac{1}{\omega_k-\omega}\d\omega = c\frac{\omega_k^\alpha\log (\omega_k^q)}{\alpha_k}\leq c\frac{\log \alpha_k}{\alpha_k^{1-\alpha/2}},
\end{aligned}
\end{equation*}
since $\omega_k^2/\alpha_k=O(1)$ as $k\to\infty$, by Lemma~\ref{lem:complete-monotone:fourier-analysis}~(c). 

A similar argument (by writing $\omega-\omega_k$ instead of $\omega_k-\omega$ wherever is applicable) yields the estimate
\begin{align*}
I_4=\int_{\omega_k+1}^{\omega_k+\omega_k^q}\frac{\omega^\alpha\Kcos(\omega)}{\alpha_k^2\Kcos^2(\omega)+\omega^2\left(\alpha_k\frac{\Ksin(\omega)}{\omega}-1\right)^2}\d\omega\leq c\frac{\log \alpha_k}{\alpha_k^{1-\alpha/2}}.
\end{align*}

To estimate $I_3$, we recall from Lemma~\ref{lem:complete-monotone:fourier-analysis}~(b) that $\omega^2\Kcos(\omega)$ is increasing on $\omega\in(0,\infty)$. It follows that
\begin{align*}
I_3=\int_{\omega_k-1}^{\omega_k+1}\frac{\omega^\alpha\Kcos(\omega)}{\alpha_k^2\Kcos^2(\omega)+\omega^2\left(\alpha_k\frac{\Ksin(\omega)}{\omega}-1\right)^2}\d\omega& \leq\int_{\omega_k-1}^{\omega_k+1}\frac{\omega^\alpha}{\alpha_k^2\Kcos(\omega)}\d\omega\\
&\leq \int_{\omega_k-1}^{\omega_k+1}\frac{\omega^{\alpha+2}}{\alpha_k^2\Kcos(1)}\d\omega\\
&\leq \frac{2(\omega_k+1)^{\alpha+2}}{\alpha_k^2\Kcos(1)}\\
&\leq \frac{c}{\alpha_k^{1-\alpha/2}}.
\end{align*}

We finally collect everything to arrive at
\begin{equation}\label{ineq:spatial-gle:mode:Holder:1}
\begin{aligned}
\int_0^\infty \frac{\omega^\alpha\Kcos(\omega)}{\alpha_k^2\Kcos^2(\omega)+\left(\omega-\alpha_k\Ksin(\omega)\right)^2}\d\omega &\leq c\bigg(\frac{1}{\alpha_k^{1+q-\alpha/2}}+\frac{1}{\alpha_k^{q-\alpha/2}}+\frac{\log\alpha_k}{\alpha_k^{1-\alpha/2}}\bigg)\leq \frac{c}{\alpha_k^{q-\alpha/2}},
\end{aligned}
\end{equation}
which holds for sufficiently large $k$, since $\alpha_k\uparrow\infty$ as $k\to\infty$ by Assumption~\ref{cond:A}. We therefore conclude~\eqref{ineq:spatial-gle:mode:Holder}. The proof is thus complete.
\end{proof}

Finally, we give the proof of Corollary~\ref{cor:Holder_space}.
\begin{proof}[Proof of Corollary~\ref{cor:Holder_space}]
The argument for~\eqref{ineq:spatial-gle:mode:Holder:2:k<k^*} is omitted as it is similar to that for~\eqref{ineq:spatial-gle:mode:Holder:k<k^*} as in the proof of Proposition~\ref{prop:spatial-gle:mode:Holder}. 

With regard to~\eqref{ineq:spatial-gle:mode:Holder:2}, we have 
\begin{equation*}
\begin{aligned}
\Enone|u_k(t)|^2&=\frac{\lambda_k^2}{\pi}\int_0^\infty\close\frac{\Kcos(\omega)}{\alpha_k^2\Kcos^2(\omega)+\left(\omega-\alpha_k\Ksin(\omega)\right)^2}\d\omega\\
&=\frac{\lambda_k^2}{\pi}\Big\{\int_0^1+\int_1^\infty\Big\} \frac{\Kcos(\omega)}{\alpha_k^2\Kcos^2(\omega)+\left(\omega-\alpha_k\Ksin(\omega)\right)^2} \d\omega.
\end{aligned}
\end{equation*}
To bound the first integral on the RHS we observe:
\begin{align*}
\MoveEqLeft[5]\int_0^1\close\frac{\Kcos(\omega)}{\alpha_k^2\Kcos^2(\omega)+\left(\omega-\alpha_k\Ksin(\omega)\right)^2}\d\omega\\
& \leq \int_0^1\close\frac{1}{\alpha_k^2\Kcos(\omega)}\d\omega \leq \int_0^1\frac{1}{\alpha_k^2\Kcos(1)}\d\omega= \frac{1}{\alpha_k^2\Kcos(1)},
\end{align*}
since $\Kcos(\omega)$ is decreasing on $\omega\in(0,\infty)$, by Lemma~\ref{lem:complete-monotone:fourier-analysis}~(a). To estimate the second integral, we pick $r_1,\,r_2\in(0,1)$ such that $r_1-r_2/2=q$. We then invoke Inequality~\eqref{ineq:spatial-gle:mode:Holder:1} as follows:
\begin{align*}
\int_1^\infty\close\close\frac{\Kcos(\omega)}{\alpha_k^2\Kcos^2(\omega)+\left(\omega-\alpha_k\Ksin(\omega)\right)^2}\d\omega& \leq \int_1^\infty\close\close\frac{\omega^{r_2}\Kcos(\omega)}{\alpha_k^2\Kcos^2(\omega)+\left(\omega-\alpha_k\Ksin(\omega)\right)^2}\d\omega\\
& \leq \frac{c}{\alpha_k^{r_1-r/2}}=\frac{c}{\alpha_k^q}.
\end{align*}
We combine these two estimates to obtain~\eqref{ineq:spatial-gle:mode:Holder:2}, which concludes the proof.
\end{proof}

\section{Discussion} \label{sec:discuss}

We have rigorously analyzed a stochastic integro--partial--differential equation with memory~\eqref{eq:spatial-gle} satisfying the Fluctuation--Dissipation relationship that arises from statistical mechanical considerations in the study of thermally fluctuating viscoelastic media. Using the framework of generalized stationary processes from \cite{ito1954stationary,mckinley2018anomalous}, we obtain stationary solutions of ~\eqref{eq:spatial-gle} when the memory belongs to a large subclass $\CM_b$ of the completely monotone functions. Furthermore, we establish space--time H\"older regularity of the solutions. As we demonstrate below, when we compare the stochastic heat equation with memory to the classical formulation, the noise structure arising from the Fluctuation--Dissipation relationship yields greater regularity in time. 

The form of the equations studied here was directly motivated by the work of \cite{hohenegger2017fluid} on thermally fluctuating viscoelastic fluids. In that work, only a finite number of Fourier modes were used to define the space--time noise, and it is natural to ask when \eqref{eq:spatial-gle} is well--posed if infinitely many Fourier modes are used. The result we present is quite general and requires only Assumption \ref{cond:alpha-k:wellposed}, which is commonly seen in the linear SPDE literature~\cite{bonaccorsi2012asymptotic,
da1996ergodicity,da2014stochastic}. It is worth noting though, that in \cite{hohenegger2017fluid}, a particular form of the memory kernel was used, namely, a finite sum of exponentials. As we demonstrated, the well--posedness result that we obtain, Theorem~\ref{thm:spatial-gle:well-posed}, is applicable to a subclass of completely monotone functions of which sum-of-exponentials functions are members. We however remark that the sum--of--exponential form is not an artificial or highly restrictive one. Members of this family can approximate the class of completely monotone functions in such a way that the GLE has what is sometimes called \emph{transient anomalous diffusion}, which is to say that the associated processes are subdiffusive over arbitrarily large time intervals despite being diffusive in the large time limit.

Before moving on to the application to stochastic heat equations, we remark that our notion of solution, as well as the subsequent analysis, relies heavily on the linear structure of \eqref{eq:spatial-gle}. It also uses explicit calculations that exploit the Fluctuation--Dissipation form. It remains an open--ended question to explore well--posedness and (more interestingly) H\"older regularity of solutions when Fluctuation--Dissipation is not assumed.  Well--posedness becomes even more of a question when one considers non--linear terms to encode, for example, external forces acting on the fluid. We consider this to be an important open question.

We now discuss the regularity in the case that $A$ is the usual Laplacian operator in $\rbb^d$ with Dirichlet boundary condition on $\domain$. For the reader's convenience, we first recall the following stochastic heat equation for $u(t,\xbf):[0,\infty)\times\domain\to\rbb^d$
\begin{align}\label{eq:heat-equation:classic}
\dot{u}(t,\xbf)  = A\,u(t,\xbf)+\dot{W}(t,\xbf),\qquad (t,\xbf)\in\rbb\times\domain.
\end{align}
Here $W(t,\xbf)$ is a cylindrical Wiener process with the decomposition
\begin{align*}
W(t,\xbf)=\sum_{k\geq 1}\lambda_k W_k(t)e_k(\xbf),
\end{align*}
where $\{W_k\}_{k\geq 1}$ are i.i.d standard Brownian motions and $\{\lambda_k\}_{k\geq 1}$ are as in Assumption~\ref{cond:alpha-k:wellposed}. It is known that \cite[Example 5.24]{da2014stochastic} there exists a modification $U(t,\xbf)$ of $u(t,\xbf)$ the solution of~\eqref{eq:heat-equation:classic} such that $U(t,\xbf)$ is $\gamma-$H\"older continuous in time for $\gamma\in(0,(1-\eta)/2)$ and in space for $\gamma\in(0,1-\eta)$ where $\eta$ is as in Assumption~\ref{cond:alpha-k:wellposed}. Particularly, in the case of 1D heat equation with white noise ($\lambda_k=1$ for every $k$), the pair of H\"older constants is $(1/4,1/2)$ in $(t,\xbf)$ \cite[Page 6]{hairer2009introduction}.

Alternatively, we consider~\eqref{eq:heat-equation:classic} with memory as follows.
\begin{align}\label{eq:heat-equation:daprato}
\dot{u}(t,\xbf)  = k_0A\,u(t,\xbf)-\int_{-\infty}^t \close K(t-s)A\, u(s,\xbf)\d s+\dot{W}(t),\quad (t,\xbf)\in\rbb\times\domain,
\end{align}
where $k_0$ is a positive constant such that $k_0>\int_0^\infty K(t)\d t$ and $K$ is a completely monotone function. It was shown in~\cite[Lemma 3.7]{bonaccorsi2012asymptotic} that under Assumption~\ref{cond:alpha-k:wellposed}, the solution of~\eqref{eq:heat-equation:daprato} has the same H\"older regularity as the solution of~\eqref{eq:heat-equation:classic}. In other words, with the same noise but adding a small memory effect,~\eqref{eq:heat-equation:daprato} does not differ from~\eqref{eq:heat-equation:classic} in terms of regularity. Intuitively, this invariance can be explained as the memory effect being dominated by the dissipation because of the assumption $k_0>\int_0^\infty K(t)\d t$. We note that this condition also requires that $K$ be integrable. 

In contrast, recalling our system~\eqref{eq:spatial-gle}
\begin{align*} 
\dot{u}(t,\xbf)& =\int_{-\infty}^t \close K(t-s)A\, u(s,\xbf)\d s+ \Fbf(t,\xbf),\qquad (t,\xbf)\in\rbb\times\domain,\\
\Fbf(t,\xbf)&=\sum_{k\geq 1}\lambda_k F_k(t)e_k(\xbf),\quad\text{ and }\quad\E[F(t)F(s)]=K(|t-s|),
\end{align*}
where $K$ is not necessarily integrable, cf. Assumption~\ref{cond:K}. In view of Theorem~\ref{thm:spatial-gle:Holder}, while spatial regularity is the same as in the case of~\eqref{eq:heat-equation:classic} and~\eqref{eq:heat-equation:daprato},~\eqref{eq:spatial-gle} enjoys better regularity in time, namely $\gamma-$H\"older continuity holds in time for $\gamma\in (0,1-\eta)$. Particularly, in 1D, when $\lambda_k=1$, $k=1,2,\dots$, the pair of H\"older constants for~\eqref{eq:spatial-gle} is $(1/2-\ep,1/2-\ep)$ in $(t,\xbf)$ for every $\ep \in(0,1/2)$.

\section*{Acknowledgement}  The authors would like to thank Gustavo Didier for fruitful discussions on the topic of contour integrals. SM is grateful for support through grant NSF DMS--1644290.

\bibliographystyle{plain}
{\footnotesize\bibliography{spatial-gle-bib}}

\begin{thebibliography}{10}

\bibitem{barbu1975nonlinear}
Viorel Barbu.
\newblock {Nonlinear Volterra equations in a Hilbert space}.
\newblock {\em SIAM Journal on Mathematical Analysis}, 6(4):728--741, 1975.

\bibitem{barbu1976nonlinear}
Viorel Barbu.
\newblock {\em {Nonlinear semigroups and differential equations in Banach
  spaces}}.
\newblock Springer, 1976.

\bibitem{bonaccorsi2012asymptotic}
Stefano Bonaccorsi, Giuseppe Da~Prato, and Luciano Tubaro.
\newblock Asymptotic behavior of a class of nonlinear stochastic heat equations
  with memory effects.
\newblock {\em SIAM Journal on Mathematical Analysis}, 44(3):1562--1587, 2012.

\bibitem{bonaccorsi2004large}
Stefano Bonaccorsi and Marco Fantozzi.
\newblock {Large deviation principle for semilinear stochastic Volterra
  equations}.
\newblock {\em Dynamic Systems and Applications}, 13:203--220, 2004.

\bibitem{bonaccorsi2006infinite}
Stefano Bonaccorsi and Marco Fantozzi.
\newblock {Infinite dimensional stochastic Volterra equations with dissipative
  nonlinearity}.
\newblock {\em Dynamic Systems and Applications}, 15(3/4):465, 2006.

\bibitem{chekroun2012invariant}
Micka{\"e}l~D Chekroun and Nathan~E Glatt-Holtz.
\newblock {Invariant measures for dissipative dynamical systems: Abstract
  results and applications}.
\newblock {\em Communications in Mathematical Physics}, 316(3):723--761, 2012.

\bibitem{clement1988existence}
Philippe Cl{\'e}ment and Giuseppe Da~Prato.
\newblock {Existence and regularity results for an integral equation with
  infinite delay in a Banach space}.
\newblock {\em Integral Equations and Operator Theory}, 11(4):480--500, 1988.

\bibitem{clement1996some}
Philippe Cl{\'e}ment and Giuseppe Da~Prato.
\newblock {Some results on stochastic convolutions arising in Volterra
  equations perturbed by noise}.
\newblock {\em Atti della Accademia Nazionale dei Lincei. Classe di Scienze
  Fisiche, Matematiche e Naturali. Rendiconti Lincei. Matematica e
  Applicazioni}, 7(3):147--153, 1996.

\bibitem{clement1997white}
Philippe Cl{\'e}ment and Giuseppe Da~Prato.
\newblock White noise perturbation of the heat equation in materials with
  memory.
\newblock {\em Dynamic Systems and Applications}, 6:441--460, 1997.

\bibitem{clement1998white}
Philippe Cl{\'e}ment, Giuseppe Da~Prato, and Jan Pr{\"u}ss.
\newblock White noise perturbation of the equations of linear parabolic
  viscoelasticity.
\newblock {\em Rend. Istit. Mat. Univ. Trieste}, 29:207--220, 1997.

\bibitem{conti2020nonclassical}
Monica Conti, Filippo Dell'Oro, and Vittorino Pata.
\newblock Nonclassical diffusion with memory lacking instantaneous damping.
\newblock {\em Communications on Pure \& Applied Analysis}, 19(4):2035, 2020.

\bibitem{conti2006singular}
Monica Conti, Vittorino Pata, and Marco Squassina.
\newblock Singular limit of differential systems with memory.
\newblock {\em Indiana University mathematics journal}, pages 169--215, 2006.

\bibitem{cramer2013stationary}
Harald Cram{\'e}r and M~Ross Leadbetter.
\newblock {\em {Stationary and related stochastic processes: Sample function
  properties and their applications}}.
\newblock Courier Corporation, 2013.

\bibitem{da1996ergodicity}
Giuseppe Da~Prato and Jerzy Zabczyk.
\newblock {\em Ergodicity for infinite dimensional systems}, volume 229.
\newblock Cambridge University Press, 1996.

\bibitem{da2014stochastic}
Giuseppe Da~Prato and Jerzy Zabczyk.
\newblock {\em Stochastic equations in infinite dimensions}.
\newblock Cambridge university press, 2014.

\bibitem{didier2020asymptotic}
Gustavo Didier and Hung~D Nguyen.
\newblock Asymptotic analysis of the mean squared displacement under fractional
  memory kernels.
\newblock {\em SIAM Journal on Mathematical Analysis}, 52(4):3818--3842, 2020.

\bibitem{didier2021generalized}
Gustavo Didier and Hung~D Nguyen.
\newblock {The generalized Langevin equation in harmonic potentials: Anomalous
  diffusion and equipartition of energy}.
\newblock {\em arXiv preprint arXiv:2103.05089}, 2021.

\bibitem{weinan2002gibbsian}
Weinan E and Di~Liu.
\newblock Gibbsian dynamics and invariant measures for stochastic dissipative
  pdes.
\newblock {\em Journal of Statistical Physics}, 108(5-6):1125--1156, 2002.

\bibitem{weinan2001gibbsian}
Weinan E, Jonathan~C. Mattingly, and Yakov Sinai.
\newblock {Gibbsian Dynamics and Ergodicity for the Stochastically Forced
  Navier--Stokes Equation}.
\newblock {\em Communications in Mathematical Physics}, 224(1):83--106, 2001.

\bibitem{gatti2005navier}
Stefania Gatti, Claudio Giorgi, and Vittorino Pata.
\newblock {Navier--Stokes limit of Jeffreys type flows}.
\newblock {\em Physica D: Nonlinear Phenomena}, 203(1-2):55--79, 2005.

\bibitem{goychuk2012viscoelastic}
Igor Goychuk.
\newblock {Viscoelastic subdiffusion: Generalized Langevin equation approach}.
\newblock {\em Advances in Chemical Physics}, 150:187, 2012.

\bibitem{goychuk2014molecular}
Igor Goychuk, Vasyl~O Kharchenko, and Ralf Metzler.
\newblock How molecular motors work in the crowded environment of living cells:
  coexistence and efficiency of normal and anomalous transport.
\newblock {\em PLoS One}, 9(3):e91700, 2014.

\bibitem{hairer2009introduction}
Martin Hairer.
\newblock {An introduction to stochastic PDEs}.
\newblock {\em arXiv preprint arXiv:0907.4178}, 2009.

\bibitem{hohenegger2017equipartition}
Christel Hohenegger.
\newblock {On equipartition of energy and integrals of Generalized Langevin
  Equations with generalized Rouse kernel}.
\newblock {\em Communications in Mathematical Sciences}, 15(2):539--554, 2017.

\bibitem{hohenegger2017fluid}
Christel Hohenegger and Scott~A McKinley.
\newblock Fluid--particle dynamics for passive tracers advected by a thermally
  fluctuating viscoelastic medium.
\newblock {\em Journal of Computational Physics}, 340:688--711, 2017.

\bibitem{ito1954stationary}
Kiyosi It{\^o}.
\newblock Stationary random distributions.
\newblock {\em Memoirs of the College of Science, University of Kyoto. Series
  A: Mathematics}, 28(3):209--223, 1954.

\bibitem{ito1964stationary}
Kiyosi It{\^o} and Makiko Nisio.
\newblock On stationary solutions of a stochastic differential equation.
\newblock {\em J. Math. Kyoto Univ.}, 4(3):1--75, 1964.

\bibitem{kou2008stochastic}
Samuel~C. Kou.
\newblock Stochastic modeling in nanoscale biophysics: subdiffusion within
  proteins.
\newblock {\em The Annals of Applied Statistics}, pages 501--535, 2008.

\bibitem{kubo1966fluctuation}
Rep Kubo.
\newblock The fluctuation-dissipation theorem.
\newblock {\em Reports on progress in physics}, 29(1):255, 1966.

\bibitem{mason1995optical}
Thomas~G Mason and DA~Weitz.
\newblock Optical measurements of frequency-dependent linear viscoelastic
  moduli of complex fluids.
\newblock {\em Physical review letters}, 74(7):1250, 1995.

\bibitem{mckinley2018anomalous}
Scott~A McKinley and Hung~D Nguyen.
\newblock {Anomalous diffusion and the generalized Langevin equation}.
\newblock {\em SIAM Journal on Mathematical Analysis}, 50(5):5119--5160, 2018.

\bibitem{morgado2002relation}
Rafael Morgado, Fernando~A Oliveira, G~George Batrouni, and Alex Hansen.
\newblock Relation between anomalous and normal diffusion in systems with
  memory.
\newblock {\em Physical review letters}, 89(10):100601, 2002.

\bibitem{mori1965continued}
Hazime Mori.
\newblock A continued-fraction representation of the time-correlation
  functions.
\newblock {\em Progress of Theoretical Physics}, 34(3):399--416, 1965.

\bibitem{soni1975slowly}
Kusum Soni and Raj~P. Soni.
\newblock {Slowly varying functions and asymptotic behavior of a class of
  integral transforms I}.
\newblock {\em Journal of Mathematical Analysis and Applications},
  49(1):166--179, 1975.

\bibitem{strichartz2003guide}
Robert~S Strichartz.
\newblock {\em {A guide to distribution theory and Fourier transforms}}.
\newblock World Scientific Publishing Company, 2003.

\bibitem{yamazaki2019gibbsianB}
Kazuo Yamazaki.
\newblock Gibbsian dynamics and ergodicity of magnetohydrodynamics and related
  systems forced by random noise.
\newblock {\em Stochastic Analysis and Applications}, 37(3):412--444, 2019.

\bibitem{yamazaki2019gibbsianA}
Kazuo Yamazaki.
\newblock Gibbsian dynamics and ergodicity of stochastic micropolar fluid
  system.
\newblock {\em Applied Mathematics \& Optimization}, 79(1):1--40, 2019.

\end{thebibliography}
\end{document}